\documentclass[a4paper,10pt]{amsart}

\subjclass[2020]{Primary 16E50, 16U40} 

\keywords{Enoch's conjecture, Mittag-Leffler modules, almost free, $\mathrm{Add}(X)$}

\usepackage{amssymb}
\usepackage{amsmath}
\usepackage{amsthm}
\usepackage{color}
\usepackage{enumerate}
\usepackage{enumitem}
\newcommand{\rest}{\upharpoonright}

\newcommand{\Ext}{\operatorname{Ext}}

\newcommand{\Ker}{\operatorname{Ker}}

\newcommand{\Img}{\operatorname{Im}}

\newcommand{\Coker}{\operatorname{Coker}}

\DeclareMathOperator{\cf}{cf}

\newcommand{\C}{\mathbf{C}}

\DeclareMathOperator{\s}{s}

\newcommand{\Add}{\mathrm{Add}}

\theoremstyle{plain}
\newtheorem{theorem}{Theorem}[section]
\newtheorem{lemma}[theorem]{Lemma}
\newtheorem*{conjecture}{Conjecture}
\newtheorem{proposition}[theorem]{Proposition}
\newtheorem{corollary}[theorem]{Corollary}

\newtheorem{fact}[theorem]{Fact}

\theoremstyle{definition}
\newtheorem{definition}[theorem]{Definition}

\newtheorem{notation}[theorem]{Notation}
\newtheorem{question}[theorem]{Question}

\theoremstyle{remark}
\newtheorem{remark}[theorem]{Remark}

\newtheorem{example}[theorem]{Example}

\def\s{\subseteq}
\usepackage{tikz}
\usepackage{tikz-cd}
\newtheorem{claim}{Claim}[theorem]
\usepackage[colorlinks,citecolor=blue,urlcolor=black,linkcolor=blue]{hyperref}

\title[Almost free, perfect decomposition and Enoch's conjecture]{Almost free modules, perfect decomposition and Enochs's conjecture}

\author[Cort\'es-Izurdiaga]{Manuel Cort\'es-Izurdiaga}
\address[Cort\'es-Izurdiaga]{Departamento de Matemática Aplicada, Universidad de Málaga, 29071, Málaga, Spain}
\email{mizurdiaga@uma.es}
\author[Poveda]{Alejandro Poveda}
\address[Poveda]{Harvard University, Department of Mathematics and Center of Mathematical Sciences and Applications, Cambridge (MA), 02138, USA}
\email{alejandro@cmsa.fas.harvard.edu}
\thanks{The first author is partially supported by the Spanish Government under grants PID2020-113552GB-I00 which include FEDER funds of the EU., and by Junta de Andaluc\'{\i}a under grant P20-00770. The second author acknowledges support from the Department of Mathematics at Harvard University as well as from the Harvard Center of Mathematical Sciences and Applications.}

\begin{document}

\begin{abstract}
Given a module $X$ and a regular cardinal $\kappa$ we study various notions of $(\kappa,\Add(X))$-freeness and $(\kappa,\Add(X))$-separability. Bearing on appropriate set-theoretic assumptions, we construct a non-trivial $\kappa^+$-generated, $(\kappa^+,\Add(X))$-free and $(\kappa^+,\Add(X))$-separable module. Our construction allows $\kappa$ to be singular thus extending \cite[Theorem~4.7]{CortesGuilTorrecillas}. Bearing on similar set-theoretic assumptions, we characterize when every module $X$ has a perfect decomposition. As a subproduct  we show that Enoch's conjecture for classes $\Add(X)$ is consistent with ZFC -- a fact first proved by \v{S}aroch \cite{Saroch}. %We also generalize \cite[Theorem~4.7]{CortesGuilTorrecillas} by constructing an almost $(\aleph_{\omega+1},\Add(X))$-separable module. % that is non-trivial.
\end{abstract}

%\begin{amscode}33C50\end{amscode}

\maketitle

%\tableofcontents

\section{Introduction}\label{Introduction}
An infinite abelian group $G$ is called \emph{almost free} whenever  every subgroup of strictly smaller cardinality is free yet $G$ itself fails to be free. A natural question in infinite abelian group theory is which infinite cardinals accommodate almost free groups  \cite{Fuchs}.  Almost free groups where introduced  back in the late 30's by Kurosch \cite{Kurosch} who employed the terminology \emph{locally free groups}. Kurosch himself provided the first
construction of an almost free group of cardinality $\aleph_0$. A decade later, Higman \cite{Higman} produced an almost free group of cardinality $\aleph_1$. Ever since, major algebrist and set theorist have devoted stubborn efforts to produce a rich   and fine-structured theory of almost freeness  \cite{Shelahcompactness,Eklof, EklofMekler, Saroch}. The existence of almost free groups is an instance of a more general phenomenon called \emph{compactness} \cite{MagidorShelah}. In turn, the study of compactness is of paramount importance in infinitary combinatorics. This explains the narrow historical ties between the study of almost freeness, stationary reflection and large cardinal among others. 

\smallskip

Given $R$ a unital (non-necessarily commutative) ring  one may inquire whether this supports almost free (left $R$-)modules. The first thing to be made precise is what is meant for a module to be almost free. Indeed, if $R$ fails to be hereditary then there might be free modules none of whose non-zero submodules  are free -- thus the natural definition is faulty.  This issue is addressed in Eklof-Mekler's book \cite[Ch. IV]{EklofMekler} where a module $M$ is said to be \emph{$\kappa$-free} provided it carries a directed system of ${<}\kappa$-generated free submodules that is  \emph{big} in certain set-theoretic sense. Thus a \emph{module is almost free} whenever it is $\kappa$-generated and $\kappa$-free in the above-described sense.  In \cite[Ch. IV and VII]{EklofMekler} the authors offer an extensive study of almost freeness. Weakening of almost freeness were more recently considered by  Herbera and Trlifaj in their study of Mittag-Leffler modules \cite{HerberaTrlifaj}.

\smallskip

In this paper we investigate a notion of $\kappa$-freeness (namely, \emph{ $(\kappa,\Add(X))$-freeness}) relative to classes of the form $\Add(X)$. For a module $X$ the class $\Add(X)$ consists (modulo isomorphism)  of all modules  that are direct summands of $X^{(\mu)}$ for some cardinal $\mu$. Classes of this form subsume (among others) the projective and pure-projective modules. Thus  $(\kappa,\Add(X))$-freeness  yields a natural extension of Eklof-Mekler  $\kappa$-freeness. Notions of freeness akin to ours were previously considered by Herbera-Trlifaj \cite{HerberaTrlifaj} in the study of Mittag-Leffler modules and more recently by \v{S}aroch \cite{Saroch} in connection to the so-called  \emph{Enoch's conjecture}.

\smallskip

This paper shall also study the notion of \emph{almost separability} relative to $\Add(X)$. Recall that a module $M$ is called \emph{$\kappa$-separable} whenever every submodule $N$ of cardinality ${<}\kappa$ is contained in a ${<}\kappa$-generated free  direct summand of $M$. Assuming that $X$ is ${<}\kappa$-generated, a module $M$ will be called \emph{$(\kappa,\Add(X))$-separable} provided the  module $N$ can be taken from $\Add(X).$ This notion of separability was first studied by the first author, Guil and Torrecillas in \cite{CortesGuilTorrecillas} in connection to when certain categories of modules are closed under colimits of directed systems.

\smallskip

In the manuscript we discuss the interplay between $(\kappa,\Add(X))$-freeness and $(\kappa,\Add(X))$-separabi\-lity with  two major themes of research in  module theory.

\smallskip

The first topic regards the existence of perfect decompositions of a module $X$; to wit, when every module in $\Add(X)$ has a decomposition complementing direct summands. Assuming the Generalized Continuum Hypothesis, \cite[Corollary 5.13]{CortesGuilTorrecillas} shows that if $X$ has a perfect decomposition then there is a non-trivial\footnote{A $(\kappa,\Add(X))$-separable module is said to be trivial if  it  belongs to the class $\Add(X)$.} $(\kappa^+,\Add(X))$-separable module, being $\kappa$ the cardinality of a generating set of $X$. In this paper,  we improve the construction of \cite{CortesGuilTorrecillas} by 
showing that non-trivial $(\kappa^+,\Add(X))$-separable modules exist even for singular cardinals $\kappa$ (the construction in \cite{CortesGuilTorrecillas} only allows regular cardinals $\kappa$). Second, we characterize, modulo a set-theoretic hypothesis consistent with ZFC, when every module $X$ has a perfect decomposition. Granting this hypothesis, $X$ does not have a perfect decomposition if and only if there is a non-trivial $(\kappa^+,\Add(X))$-separable module  %prove the converse of this result using another set theoretical hypothesis concerning tree-like ladder systems on non-reflecting stationary sets 
(Theorem \ref{thm:perfectdecom}).

\smallskip

The second topic under consideration regards the so-called Enoch's Conjecture. For a class of modules $\mathcal{X}$ the conjecture asserts that $\mathcal{X}$ is  closed under direct limits provided it is covering. Just recently J. \v Saroch \cite{Saroch} has proved (modulo an extra set-theoretic hypothesis) the consistency of the conjecture  for classes of the form $\Add(X)$ \cite{Saroch}. Saroch's argument is based on the construction of a certain $(\kappa,\Add(X))$-free module. In what we are concerned,  here we establish the precise relationship between the existence of a non-trivial $(\kappa,\Add(X))$-separable module and the existence of $\Add(X)$-covers (see Lemma~\ref{ExistenceofAddXcovers}). As a bi-product we provide an alternative proof of Enoch's conjecture for $\Add(X)$ using separable modules. Like \v{S}aroch's, our argument also requires an extra set-theoretic assumption.

\smallskip

The following is a succinct account of the paper's contents. In Section~\ref{sec: prelimminaries} we provide all the pertinent module and set-theoretic prelimminaries and notations. Section~\ref{sec: free} analyzes  various notions of almost freeness relative to the class $\Add(X)$ and connects them with the main device used in the paper -- $\Add(X)$-filtrations. A key role will be played by Lemma \ref{WhenMistrivial2} which will allow us to elucidate whether an almost free module is trivial (i.e., whether it belongs to $\Add(X)$). In Section~\ref{sec: separable} we will shift our attention to $(\kappa,\Add(X))$-separable modules. Here we prove that under suitable set-theoretic assumptions there are $\kappa^+$-generated non-trivial $(\kappa^+, \Add(X))$-separable (and $(\kappa^+, \Add(X))$-free) modules, even when $\kappa$ is a singular cardinal.  %  -- this extends one of the main results in \cite{CortesGuilTorrecillas}. 
Finally, Section~\ref{sec: applications} is devoted to applications. Here we show: first, under a consistent set-theoretic assumption, the existence of a non-trivial $(\kappa,\Add(X))$-free/separable module is equivalent to $X$ not having a perfect decomposition; second, under the same assumptions, Enoch's Conjecture holds for classes of the form $\Add(X)$.  This provides an alternative proof  to \v{S}aroch's \cite{Saroch}; third, the class of all left $R$-modules with $\mathcal Q$-Mittag-Leffler dimension less than or equal to $n$ (\cite{CortesProducts}) satisfies Enoch's conjecture, thus extending \cite[Theorem 2.6]{TrlifajYassine} to the non-flat setting.

\section{Preliminaries}\label{sec: prelimminaries}

The cardinality of a set $A$ will be denoted by $|A|$ and its power set by $\mathcal P(A)$. The restriction of a map $f:A \rightarrow B$ to a subset $A'$ of $A$ will be denoted $f\rest A'$. As customary, infinite cardinals will be denoted by  Greek letters, such as $\kappa,\lambda, \delta$, etc. The cofinality of an infinite cardinal $\kappa$ will be denoted by $\cf(\kappa)$. Recall that a cardinal $\kappa$ is  \emph{regular} if $\kappa=\cf(\kappa)$; otherwise $\kappa$ is called \textit{singular}.  Let $\beta$ be a limit ordinal with uncountable cofinality. A set $C\s \beta$ is called \emph{a club} if it is closed and unbounded with respect to the order topology of $\beta$. A set $S\s \beta$ is called \emph{stationary} if $S\cap C\neq \emptyset$ for all clubs $C\s \beta$. The set of all stationary subsets of $\beta$ will be denoted by $\mathrm{NS}_\beta^+$. A typical stationary set is $E^\kappa_\lambda:=\{\alpha<\kappa\mid \mathrm{cf}(\alpha)=\lambda\}$ whenever $\lambda<\kappa$ are both regular cardinals. A stationary set $S\s \beta$  \emph{reflects} if there is $\alpha<\beta$ with $\cf(\alpha)>\omega$ such that $S\cap \alpha$ is stationary in $\alpha$. We will say that $E$ reflects at $\alpha$. %Another important set-theoretic object used here is 
 $I[\kappa^+]$ will denote \emph{Shelah's Approchability ideal} on $\kappa^+$ (cf. \cite{Shelah, Eisworth}).
\smallskip

Through the manuscript $R$ will stand for a (non-necessarily commutative) ring with unit. By a \emph{module} we mean a left $R$-module -- i.e., a member of  $R$-$\mathrm{Mod}$. 

\smallskip

Let $\kappa$ be an infinite (non-necessarily regular) cardinal. A module $M$ is called \emph{$\kappa$-generated} if it has a set of generators $G\s M$ with size $\kappa$; to wit, every $x\in M$ can be written as $\sum_{i=0}^n r_i y_i$ for $r_i\in R$, $y_i\in G$ and $n<\omega.$ Similarly, $M$ will be called \emph{${<}\kappa$-generated} if it admits  a set of generators of cardinality ${<}\kappa.$ We will denote by $\lambda_M$ the minimum of the set
$$\{\lambda\mid \text{There is a set of generators for $M$ of cardinality $\lambda$}\}.$$
Thus $M$ will be ${<}\kappa$-generated for an infinite cardinal $\kappa$ if and only if  $\lambda_M<\kappa$. Through the paper we will use $``M$ is ${<}\kappa$-generated'' and $``\lambda_X<\kappa$'' interchangeably as often times the latter phrasing becomes more convenient than the former.

The module $M$ is $\kappa$-\textit{presented} (resp. $<\kappa$-presented) if it is $\kappa$-generated (resp. $<\kappa$-generated) and has a free presentation with $\kappa$-generated kernel. 

\smallskip

Filtrations of modules will also play a prominent role in this work:

\begin{definition}\label{def: filtration}
Let $M$ be a module and $\kappa$ a regular cardinal. A \emph{$\kappa$-filtration} of $M$ is a sequence $\langle M_\alpha\mid \alpha<\kappa\rangle$ consisting of ${<}\kappa$-generated submodules of $M$ such that
\begin{enumerate}
    \item $M_\alpha\s M_{\beta}$ for all $\alpha<\beta$,
    \item $M_{\alpha}=\bigcup_{\beta<\alpha} M_\beta$ whenever $\beta$ is a limit ordinal,
    \item $M=\bigcup_{\alpha<\kappa} M_\alpha.$
\end{enumerate}
\end{definition}

The class $\Add(X)$, where $X$ is a fixed module, will have a promient role in this paper. Recall that if $\mathcal X$ is a class of modules, $\mathrm{Sum}(\Add(X))$ is the class of all direct sums of modules from $\mathcal X$ and
$$\mathrm{Add}(\mathcal X)=\{N\in \text{$R$-$\mathrm{Mod}$}\mid N\lesssim_{\oplus} X\, \text{for some }X \in \textrm{Sum}(\mathcal X) \},$$ where $N\lesssim_{\oplus} X$ is a shorthand for $``N$ is isomorphic to a direct summand of $X$''. Clearly, $\Add(\mathcal X)$ is closed both under direct summands and direct sums. If $\mathcal X$ consists of one single module, we will write $\Add(X)$ and $\textrm{Sum}(X)$ instead of $\Add(\{X\})$ and $\textrm{Sum}(\{X\})$.  %(here $\leq_{\oplus}$ stands for direct summand). 

\smallskip

Often times an infinite cardinal %\ale{Does this work for $\aleph_0$?} 
$\kappa$ will be fixed and we will consider the classes $\textrm{Sum}_\kappa(\mathcal X)$ consisting of all direct sums of less than $\kappa$ modules belonging to $\mathcal X$ and  $\mathrm{Add}_\kappa(\mathcal X)$, of all modules isomorphic to a direct summand of a module from $\textrm{Sum}_\kappa(\mathcal X)$. Notice that, if $\kappa$ is regular and $\lambda_X < \kappa$, the modules in $\Add_\kappa(X)$ are the $<\kappa$-generated modules in $\Add(X)$. For instance, $\Add(R)$ is the class of all projective modules. Similarly, if $M$ is the direct sum of all finitely presented modules up to isomorphism, $\Add(M)$ is the class of all pure-projective modules. More generally, if $\kappa$ is an uncountable regular cardinal and $M$ is the direct sum of all $<\kappa$-presented modules up to isomorphism, then $\Add(M)$ consists of all $\kappa$-pure-projective modules, i.e., those modules $N$ which are projective with respect to all $\kappa$-pure exact sequences \cite[Proposition 5.1]{CortesGuilTorrecillas}.

\smallskip 

%One classical fact about $\kappa$-generated modules in $\Add(X)$ is \emph{Eilenberg's trick}:
%\begin{fact}[Eilenber's trick, {\cite[Lemma 3.1]{GobelTrlifaj}}]\label{fact: eilenbberg's}
%    If $M \in \Add(X)$ is a $\kappa$-generated module then $M \oplus X^{(\kappa)} \cong X^{(\kappa)}.$
%\end{fact}

A classical result we shall bear on is Walker's theorem \cite[Theorem~26.1]{AndersonFuller} -- this asserts that every direct summand of a module which is a direct sum of $\kappa$-generated modules is itself a direct sum of $\kappa$-generated modules.  In particular, if $X$ is a direct sum of $\kappa$-generated modules then so is every module in $\Add(X)$. Walker's result is a generalization of Kaplansky's theorem concerning projective modules. In turn, Walker's  result can be easily obtained for $\kappa$-presented modules:
\begin{lemma}\label{l:WalkerForPresented}
Let $\kappa$ be a regular cardinal and $M$ be a module which is a direct sum of $\kappa$-presented modules. Then every direct summand of $M$ is a direct sum of $\kappa$-presented modules.
\end{lemma}

\begin{proof}
Let $K$ be a direct summand of $M$. By Walker's Theorem, $K$ is a direct sum of $\kappa$-generated modules. Note that every $\kappa$-generated direct summand $K'$ of $K$ is $\kappa$-presented: If $M=\bigoplus_{i \in I}M_i$ for a family $\{M_i \mid i \in I\}$ of $\kappa$-presented submodules of $M$,  $K' \leq \bigoplus_{j \in J}M_j$ for some subset $J$ of $I$ with cardinality $\kappa$. Then $K'$ is a direct summand of the $\kappa$-presented module $\bigoplus_{j \in J}M_j$ and it is itself $\kappa$-presented.
\end{proof}

%By $\lambda_M$ we  denote the least cardinal $\lambda$ for which there is a set of generators for $M$ of cardinality $\lambda$.

%Note that $M$ is not necessarily the direct sum of the members of its $\kappa$-filtration. However, under special circumstancies it is possible to ensure that $M$ is a direct sum of submodules of the given filtration. 
The next (truly useful) fact will be used repeatedly through the paper.  
\begin{fact}[{\cite[Lemma~1.1]{Cortes17}}]\label{fact: when M is a direct sum}
    Suppose that $\langle M_\alpha\mid \alpha<\kappa\rangle$ is a $\kappa$-filtration of a module $M$ and that for each $\alpha<\kappa$ there is $N_\alpha\leq M$ such that $M_{\alpha+1}=M_\alpha\oplus N_\alpha$. Then, $M=\bigoplus_{\alpha<\kappa} N_\alpha$.
\end{fact}

%In the forthcoming sections we will be consider various notions of \emph{freeness} and \emph{separability} relative to the class $\Add(X)$. A module $M$ will lie trivially in this classes whenever $M\in \Add(X)$. The next lemma characterizes when a module is trivial relative to $\Add(X)$:

A direct system of modules $\mathcal{M}$ where the underlying poset is well-ordered is called a \emph{well ordered system}. A class $\mathcal{X}$ of modules is called \emph{closed under direct limits} (resp. \emph{closed under well-ordered limits}) if for any direct (resp. well-ordered) system $\mathcal{M}$  consisting of modules in $\mathcal{X}$,  $\varinjlim\mathcal{M}\in\mathcal{X}.$
%It is known that if a class of modules $\mathcal X$ contains the limits of all well ordered direct systems of modules in $\mathcal X$ then it is closed under direct limits. 

\begin{lemma}\label{lemma: prelimminaries well-ordered limits}
Let $\mathcal X$ be a class of modules. The following assertions are equivalent:
\begin{enumerate}
\item $\mathcal X$ is closed under direct limits.

\item $\mathcal X$ is closed under well ordered direct limits.

\item For any infinite cardinal $\kappa$, $\mathcal X$ is closed under well ordered direct limits of cardinality greater than or equal to $\kappa$.

\item There exists an infinite cardinal $\kappa$ such that $\mathcal X$ is closed under well ordered direct limits of cardinality greater than or equal to $\kappa$.
\end{enumerate}
\end{lemma}

\begin{proof}
(1) $\Leftrightarrow$ (2). Well known.

(2) $\Rightarrow$ (3) $\Rightarrow$ (4). Clear.

(4) $\Rightarrow$ (2). Take a direct system of modules belonging to $\mathcal X$, $(X_\delta,u_{\delta\varepsilon}\mid \delta \leq \varepsilon < \lambda)$, for some cardinal $\lambda < \kappa$ and denote by $(X,u_\delta \mid \delta < \lambda)$ its direct limit. We may assume that $\lambda$ is regular since, otherwise, $\lambda$ has a cofinal subset of cardinality some regular cardinal.

Let $\mu \geq \kappa$ be a cardinal with $\cf(\mu)=\lambda$. Take a strictly increasing map $\gamma:\lambda \rightarrow \mu$ such that $\{\gamma(\delta) \mid \delta < \lambda\}$ is unbounded in $\mu$. Let $\sigma: \mu \rightarrow \lambda$ be the map defined by $\sigma(\alpha)=\sup\{\delta < \lambda \mid \gamma(\delta) \leq \alpha\}$ for any $\alpha < \mu$. Set $Y_\alpha = X_{\sigma(\alpha)}$ and $v_{\alpha\beta}=u_{\sigma(\alpha)\sigma(\beta)}$ for each $\alpha \leq  \beta < \lambda$ so that we get a well ordered direct system $(Y_\alpha,v_{\alpha\beta}\mid \alpha \leq \beta < \mu)$. Since $\sigma$ is increasing, its image $C$ is unbounded in $\lambda$ and the direct limit of $(Y_\alpha,v_{\alpha\beta}\mid \alpha \leq \beta < \mu)$ is precisely $(X,u_\delta \mid \delta \in C)$. By (4), $X$ belongs to $\mathcal X$.
\end{proof}

In the last section of this paper we will be concerned with \emph{Enoch's conjecture}. The two backbone notions underpinning this conjecture are \emph{precovering} and \emph{covering} classes of modules.
\begin{definition}
Let $\mathcal{X}$ be a class of modules. We say that $\mathcal{X}$ is:
\begin{enumerate}
    \item \emph{Precovering}: If every module $M$ posseses an \emph{$\mathcal{X}$-precover}; namely, a morphism $f\colon X\rightarrow M$ with $X\in \mathcal{X}$ such that $\mathrm{Hom}(Y,f)$ is surjective for all $Y\in\mathcal{X}.$ 
    \item \emph{Covering}: If every module $M$ posseses an \emph{$\mathcal{X}$-cover}; namely, an $\mathcal{X}$-precover $f\colon X\rightarrow M$ such that every   $g\in\mathrm{End}(X)$ is an automorphism if $f\circ g=f$.
\end{enumerate}
   %  A class of modules $\mathcal{K}$ is called \emph{precovering} if every module $M$ posseses a \emph{$\mathcal{K}$-precover}; namely, a morphism $\varphi\colon C\rightarrow M$ with $C\in \mathcal{K}$ such that $\mathrm{Hom}(D,f)$ is surjective for all $D\in\mathcal{K}.$ In addition,  $\mathcal{K}$ is called \emph{covering} if every module $M$ admits a \emph{$\mathcal{K}$-cover}; namely, a $\mathcal{K}$-precover $\varphi\colon C\rightarrow M$ such that for each $g\in\mathrm{End}(C)$, $fg=f$ is an automorphism of $C$.
\end{definition}
The key classical result bridging these concepts is Enoch's theorem \cite[Theorem~5.31]{GobelTrlifaj} asserting that any precovering class closed under direct limits is covering. The converse of this assertion is the so-called \emph{Enoch's conjecture}. Namely,
\begin{conjecture}[Enochs]
    Every covering class of modules is closed under direct limits.
\end{conjecture}

\section{$(\kappa,\Add(X))$-free modules}\label{sec: free}

Related to an infinite regular cardinal $\kappa$, there are several notions of $\kappa$-free modules which have been extensively studied in the category of abelian groups \cite{EklofMekler}: $\kappa$-free in the weak sense, $\kappa$-free, strongly $\kappa$-free and $\kappa$-separable. In this paper we are interested in these notions but relative to the class $\Add(X)$ for a fixed module $X$. We begin with the almost free modules in the weak sense which were introduced in \cite[p. 89]{EklofMekler}. We consider here the slightly more general approach of \cite[Definition 2.2]{ProductsOfFlats} which is based on \cite[Definition 2.5]{HerberaTrlifaj}:

\begin{definition}[$(\kappa,\mathcal X)$-freeness in the weak sense]\label{d:WeakFree}
    Let $\kappa$ be an uncountable regular cardinal and $\mathcal X$, a class of modules. We say that a module $M$ is \textit{$(\kappa,\mathcal X)$-free in the weak sense} if there is a direct system $\mathcal S$ of submodules of $M$ belonging to $\mathcal X$ such that any subset of $M$ of cardinality smaller than $\kappa$ is contained in an element of $\mathcal S$. We say that $M$ is trivial if $M$ belongs to $\mathcal X$.
\end{definition}

Notice that in this definition the modules in $\mathcal S$ need not be $<\kappa$-generated as in the definition by Eklof and Mekler. But if we take $\mathcal X$ the class of all $<\kappa$-generated free modules, then the $(\kappa,\mathcal X)$-free modules in the weak sense are the $\kappa$-free modules in the weak sense of Eklof and Mekler.

\smallskip

The relationship between weak almost freeness and filtrations is the following:

\begin{proposition}\label{p:WeakFreeAndFiltrations}
    Let $X$ be a module and $\kappa$ an uncountable regular cardinal. Let $M$ be any module.
    \begin{enumerate}
        \item If $M$ has a filtration, $( M_\alpha \mid \alpha < \kappa )$, with $M_{\alpha+1} \in \Add(X)$ for each $\alpha < \kappa$, then $M$ is $(\kappa,\Add(X))$-free in the weak sense.
        
        \item If $\kappa > \lambda_X$, $M$ is $\kappa$-generated and is $(\kappa,\Add(X))$-free in the weak sense, then it has a $\kappa$-filtration, $( M_\alpha \mid \alpha < \kappa )$, with $M_{\alpha+1} \in \Add_\kappa(X)$ for each $\alpha < \kappa$. In particular, $M$ is $(\kappa,\Add_\kappa(X))$-free in the weak sense.
    \end{enumerate}
\end{proposition}

\begin{proof}
    (1) Just notice that $\{M_{\alpha+1}\mid \alpha < \kappa\}$ is a direct system satisfying the conditions in Definition \ref{d:WeakFree}.

    (2) Let $\{m_\alpha \mid \alpha < \kappa\}$ be a generating system of $M$ and $\mathcal S$ a direct system of sumodules of $M$ witnessing the $(\kappa,\Add(X))$-freeness of $M$. We can construct the modules $M_\alpha$ in the $\kappa$-filtration satisfying $m_\alpha \in M_{\alpha+1}$ recursively on $\alpha$: Set $M_0=0$ and, for $\alpha < \kappa$ limit, $M_\alpha=\bigcup_{\gamma < \alpha}M_\gamma$. For the case $\alpha+1$, take a submodule $N$ of $M$ belonging to $\mathcal S$ and containing $M_\alpha \cup \{x_\alpha\}$. Since $N$ is a direct sum of $\lambda_X$-generated modules belonging to $\Add(X)$ by Walker's lemma, $\lambda_X < \kappa$, and $\kappa$ is regular, we can find a $<\kappa$-generated direct summand $M_{\alpha+1}$ of $N$ containing $M_\alpha \cup \{x_\alpha\}$. This concludes the construction.
\end{proof}

One of the main topics discussed in this paper is when an almost free module relative to $\Add(X)$ is trivial. As the next lemma evidences, stationary sets play an important role when it comes to this issue:

\begin{lemma}\label{WhenMistrivial2}
Let $X$ a module and $\kappa$, an uncountable regular cardinal with $\lambda_X < \kappa$. Let $M$ be a $\kappa$-generated and $(\kappa,\Add(X))$-free module in the weak sense. Then the following assertions are equivalent:
     \begin{enumerate}
         \item  $M\in \mathrm{Add}(X)$.
         
         \item $M$ is a direct sum of $<\kappa$-generated modules.
         
         \item For any $\kappa$-filtration of $M$, $\langle M_\alpha \mid \alpha < \kappa\rangle$, with $M_{\alpha+1} \in \Add(X)$ for every $\alpha < \kappa$, the set $$E=\{\alpha<\kappa\mid M_\alpha\nleq_\oplus M\}$$ is not stationary in $\kappa$.

         \item There is a $\kappa$-filtration of $M$, $\langle M_\alpha \mid \alpha < \kappa\rangle$, with $M_{\alpha+1} \in \Add(X)$ for every $\alpha < \kappa$ such that the set $$E=\{\alpha<\kappa\mid M_\alpha\nleq_\oplus M\}$$ is not stationary in $\kappa$.

         \item For any $\kappa$-filtration of $M$, $\langle M_\alpha \mid \alpha < \kappa\rangle$, with $M_{\alpha+1} \in \Add(X)$ for every $\alpha < \kappa$, the set $$E'=\{\alpha<\kappa\mid \{\beta > \alpha \mid M_\alpha \nleq_\oplus M_\beta\} \textrm{ is stationary in }\kappa\}$$ is not stationary in $\kappa$.

         \item There is a $\kappa$-filtration of $M$, $\langle M_\alpha \mid \alpha < \kappa\rangle$, with $M_{\alpha+1} \in \Add(X)$ for every $\alpha < \kappa$ such that the set $$E'=\{\alpha<\kappa\mid \{\beta > \alpha \mid M_\alpha \nleq_\oplus M_\beta\} \textrm{ is stationary in }\kappa\}$$ is not stationary in $\kappa$.
     \end{enumerate}
\end{lemma}

\begin{proof}
    (1) $\Rightarrow$ (2). If $M \in \Add(X)$, then $M$ is a direct sum of $\lambda_X$-generated modules by Walker's lemma \cite[Theorem~26.1]{AndersonFuller}.

    \smallskip

    (2) $\Rightarrow$ (3). Let $\langle M_\alpha \mid \alpha < \kappa\rangle$ be a $\kappa$-filtration of $M$ with $M_\alpha \in \Add(X)$ for every $\alpha < \kappa$. Write $M=\bigoplus_{\alpha < \kappa}X_\alpha$ for suitable $X_\alpha \in \Add_\kappa(X)$. Setting $N_\alpha = \bigoplus_{\gamma < \alpha}X_\gamma$ we get a $\kappa$-filtration of $M$ with $N_\alpha \leq_\oplus M$ for each $\alpha < \kappa$. By \cite[Lemma IV.1.4]{EklofMekler}, there is club $C$ such that $M_\alpha = N_\alpha$ for each $\alpha < \kappa$. Then $C \cap E=\emptyset$ and $E$ is not stationary.

     \smallskip

    (3) $\Rightarrow$ (4). Trivial.

    \smallskip

    (4) $\Rightarrow$ (1). Let $C\s \kappa$ be a club disjoint from $E$. Let $C(\cdot)\colon \kappa\rightarrow C$ be an order-preserving continious bijection and set $N_\alpha:=M_{C(\alpha)}$. Clearly, $\langle N_\alpha\mid \alpha<\kappa\rangle$ is another $\kappa$-filtration of $M$  consisting of members of $\Add(X)$ which in addition satisfies that $N_\alpha=M_{C(\alpha)}\leq_{\oplus} M_{C(\alpha)+1}=N_{\alpha+1}$. Hence,  $N_{\alpha+1}=N_\alpha\oplus N_{\alpha+1}'$ for all $\alpha<\kappa$. By Fact~\ref{fact: when M is a direct sum}, $\textstyle M=\bigoplus_{\alpha<\kappa} N'_{\alpha+1}$ which means that $M\in \mathrm{Add}(X)$.%\man{Order preserving implica continua?}\ale{No. Lo acabo de corregir. Gracias.}

    (4) $\Rightarrow$ (5). Follows from the fact that $E \subseteq E'$.

    (5) $\Rightarrow$ (6). Trivial.

    (6) $\Rightarrow$ (1). The same proof of \cite[Proposition IV.1.7]{EklofMekler} applies.
\end{proof}

Let us consider $\kappa$-freeness of modules relative to a class \cite[Definition 2.2]{ProductsOfFlats}:

\begin{definition}[$(\kappa,\mathcal X)$-freeness]\label{d:KappaFree}
%Let $\kappa$ be an uncountable cardinal and $M$, $X$, modules.
Let $\kappa$ be an uncountable regular cardinal and $\mathcal X$, a class of modules. A module $M$ is called \emph{$(\kappa,\mathcal X)$-free} if it has a $(\kappa,\mathcal X)$-dense system of submodules, i.e., a direct system $\mathcal{S}\s \mathcal{P}(M)$ consisting of submodules of $M$ such that:
%\begin{enumerate}
%\item If $\kappa$ is regular, the module $M$ is called $(\kappa,\Add(X))$-free if there exists a subset $\mathcal S$ of submodules of $M$ such that:
\begin{enumerate}
\item Every module in $\mathcal S$ belongs to the class $\mathcal X$.

\item $\mathcal S$ is closed under unions of well-ordered chains of length ${<} \kappa$.

\item Every subset of $M$ of cardinality ${<}\kappa$ is contained in an element of $\mathcal S$.
\end{enumerate}
A $(\kappa,\mathcal X)$-free module $M$ is called \emph{trivial} if $M\in \mathcal X.$
%\item If $\kappa$ is singular, $M$ is $(\kappa,\Add(X))$-free if it is $(\lambda,\Add(X))$-free for every regular cardinal $\lambda < \kappa$.
%\end{enumerate}
\end{definition}

Clearly, every $(\kappa,\mathcal X)$-free module is $(\kappa,\mathcal X)$-free in the weak sense. If $\mathcal X$ is the class of all $<\kappa$-generated free modules, then the $(\kappa,\mathcal X)$-free modules are the $\kappa$-free modules by Eklof and Mekler \cite[Definition IV.1.1]{EklofMekler}. If $\mathcal X$ is the class of all countably generated pure-projective modules, then the $(\aleph_1,\mathcal X)$-free modules are the Mittag-Leffler modules \cite[Corollary 2.7]{HerberaTrlifaj}. If $\mathcal X$ is the class of all countably generated projective modules, then the $(\aleph_1,\mathcal X)$-free modules are the flat Mittag-Leffler modules \cite[Corollary 2.10]{HerberaTrlifaj}.

\begin{remark}
It is possible to tweak the above definition to encompass the case where $\kappa$ is a singular cardinal (such as $\aleph_\omega$). In that case one stipulates  $M$ to be {$(\kappa,\Add(X))$-free} if it is $(\lambda,\Add(X))$-free for every regular cardinal $\lambda < \kappa$. By Shelah's Singular Compactness Theorem \cite[Theorem IV.3.7]{EklofMekler}, every $(\kappa,\Add(X))$-free module is trivial (i.e., it belongs to $\Add(X)$) when $\kappa$ is a singular cardinal.
\end{remark}

%By Shelah's Singular Compactness Theorem \cite[Theorem IV.3.7]{EklofMekler}, every $(\kappa,\mathcal X)$-free module is trivial when $\kappa$ is a singular cardinal (\textcolor{red}{comprobar!}).

Similarly to the classical notion of $\kappa$-freeness of \cite[IV.1.1]{EklofMekler}, $(\kappa,\Add_\kappa(X))$-freeness can be as well characterized by the existence of certain $\kappa$-filtrations: % consisting of members of $\Add(X)$.
\begin{proposition}\label{prop: filtration and freeness}
Let $\kappa$ be uncountable regular  and $M$ and $X$ be modules.

\begin{enumerate}
    \item If $M$ has a filtration, $\langle M_\alpha \mid \alpha < \kappa \rangle$, with $M_\alpha \in \Add(X)$ for each $\alpha < \kappa$, then $M$ is $(\kappa,\Add(X))$-free.

    \item If $\kappa > \lambda_X$, $M$ is $\kappa$-generated and $(\kappa,\Add_\kappa(X))$-free, then $M$ has a filtration, $\langle M_\alpha \mid \alpha < \kappa \rangle$, with $M_\alpha \in \Add_\kappa(X)$ for each $\alpha < \kappa$.
\end{enumerate}

%Thus, every ${\leq} \kappa$-generated $(\kappa,\Add(X))$-free module is weakly $(\kappa,\Add(X))$-free.
\end{proposition}

\begin{proof}
(1) Just notice that $\{M_\alpha \mid \alpha < \kappa\}$ is a $\kappa$-dense system consisting of modules belonging to $\Add(X)$.

(2) If $\mathcal S \subseteq \mathcal P(M) $ is a $(\kappa,\mathcal S)$-dense system of $M$, then we can reason as in Proposition \ref{p:WeakFreeAndFiltrations} to obtain a $\kappa$-filtration of $M$, $\langle M_\alpha \mid \alpha < \kappa \rangle$, with $M_{\alpha+1} \in \mathcal S$. But, since $\mathcal S$ is closed under well-ordered unions of length smaller than $\kappa$, $M_\alpha \in \Add_\kappa(X)$ for every limit ordinal $\alpha < \kappa$ as well.
\end{proof}

%\subsection{$(\kappa,\Add(X))$-freeness and stationary reflection}

%Recall that during this section $X$ denotes a  ${<}\kappa$-generated module.

%In the next few lemmas we will argue that if $\lambda_X^+<\kappa$ (e.g., if $\kappa$ is a limit cardinal) then the existence of non-trivial $(\kappa, \Add(X))$-free groups yield non-reflecting stationary subsets of $\kappa$. Putted in set-theoretic terms, this demonstrates  that non-trivial $(\kappa, \Add(X))$-free groups are genuine non-compact objects.

%\begin{proof}
 %   We just need to show that a weakly compact cardinal $\kappa$ satisfies $(\star)$ above. This is a rather standard observation but we include it for the benefit of our readers. 
%
 %   Recall that $\kappa$ is weakly compact if and only if it is $\Pi_1^1$-indescribable.  Consider $$\text{$\{\alpha<\kappa\mid \text{$\alpha$ is regular}\}$ and $\{\alpha<\kappa\mid \text{$S\cap \alpha$ is stationary}\}$}.$$
  %  Both sets are easily seen to belong to the $\Pi^1_1$-indescribable filter, hence so does their intersection.  Since the $\Pi^1_1$-indescribable filter is  normal (see \cite[Proposition~6.11]{Kan}) all of its members are stationary in $\kappa$. This clearly yields $(\star)$ above. 
%\end{proof}

Another interesting instance of almost freeness is the following due to Eklof and Mekler \cite[Definition IV.1.8]{EklofMekler}:

\begin{definition}[Strongly $(\kappa,\mathcal X)$-freeness]\label{d:StronglyFree}
Let $\kappa$ be an uncountable regular cardinal and $M$, a module. We say that $M$ is \textit{strongly $(\kappa,\mathcal X)$-free} if there is a set $\mathcal S \subseteq \mathcal X$ of submodules of $M$ belonging to $\mathcal X$ and containing $0$ such that for any $S \in \mathcal S$ and $Y\subseteq M$ with $|Y|<\kappa$, there exists $S' \in \mathcal S$ containing $Y \cup S$ and such that $S$ is a direct summand of $S'$.

A strongly $(\kappa,\mathcal X)$-free module $M$ is called \emph{trivial} if $M\in \mathcal X.$
\end{definition}

\begin{remark}\label{l:DirectSummand}
    Suppose that $M$ is strongly $(\kappa,\mathcal X)$-free and that $\mathcal S$ witnesses this. If $S, S' \in \mathcal S$ and $S \leq S'$ then  $S\leq_{\oplus}S'$:  By the definition of strongly free we can find $S'' \in \mathcal S$ with $S' \leq S''$ and $S$ being a direct summand of $S''$. In particular, $S$ is a direct summand of $S'$.
\end{remark}

As in the classical context of \emph{strong $\kappa$-freeness} \cite[IV.1]{EklofMekler}, a strongly $(\kappa,\Add(X))$-free module may not be $(\kappa,\Add(X))$-free (see \cite[Theorem 8]{Trlifaj} for the case $X=R$). Next we stablish the characterization of $\kappa$-generated strongly $(\kappa,\Add(X))$-free modules via filtrations. In particular, they are weakly $(\kappa,\Add(X))$-free. 

\begin{lemma}\label{l:StrongFree}
%Let $\kappa$ be uncountable regular and $M$ and $X$, modules. 
Let $\kappa$ be an uncountable cardinal and $X$ a module with $\lambda_X<\kappa$. The following assertions are equivalent for the $\kappa$-generated module $M$:
\begin{enumerate}
    \item $M$ is strongly $(\kappa,\Add(X))$-free;

    \item $M$ is strongly $(\kappa,\Add_\kappa(X))$-free; 
    
    \item $M$ admits a $\kappa$-filtration $\langle M_\alpha \mid \alpha < \kappa\rangle$ such that for $\alpha<\kappa$, 
      $M_{\alpha +1}\in \Add_\kappa(X)$  and $M_{\alpha+1}\leq_{\oplus} M_\beta$ for all $\alpha<\beta<\kappa.$
    %and being a direct summand of $M_\beta$ for every $\alpha < \beta < \kappa$.
\end{enumerate}
Thus, strongly $(\kappa,\Add(X))$-free modules are  weakly $(\kappa,\Add(X))$-free modules.
\end{lemma}

\begin{proof}
(1) $\Rightarrow$ (2). Suppose that $M$ is strongly $(\kappa,\Add(X))$-free and take the set $\mathcal S$ of submodules of $M$ given by Definition \ref{d:StronglyFree}. Let $\mathcal S'$ be the set consisting of all $<\kappa$-generated submodules of $M$ which are a direct summand of some member of $\mathcal S$. Then $\mathcal S' \subseteq \Add_\kappa(X)$ and given any $Y \subseteq M$ with $|Y| < \kappa$ and $S' \in \mathcal S$, there exists $S \in \mathcal S$ such that $S' \leq_\oplus S$. Using that $M$ is strongly $(\kappa,\Add(X))$-free, we can find $T \in \mathcal S$ with $Y \cup S \leq T$ and $S \leq_\oplus T$. Since $T$ is a direct sum of $\lambda_X$-generated modules by Walker's lemma, there exists a $<\kappa$-generated direct summand $T'$ of $T$ containing $Y \cup S'$. Then $T' \in \mathcal S'$ and, as $S' \leq_\oplus T$, $S'$ is a direct summand of $T'$ as well. This means that $\mathcal S'$ satisfies the conditions of Definition \ref{d:StronglyFree}, and that $M$ is strongly $(\kappa,\Add_\kappa(X))$-free.

(2) $\Rightarrow$ (3). Suppose that $M$ is strongly $(\kappa,\Add_\kappa(X))$-free and take $\mathcal S$ the set of submodules of $M$ given in Definition \ref{d:StronglyFree}. Repeat the argument provided in the proof of $(1)\Rightarrow (2)$ in Proposition~\ref{prop: filtration and freeness}. This yields a $\kappa$-filtration $\langle M_\alpha\mid \alpha<\kappa\rangle$ where $M_{\alpha+1}\in \mathcal{S}$ for all $\alpha<\kappa$. %Let $\{m_\alpha \mid \alpha < \kappa\}$ be a generating system of $M$. We construct the filtration $(M_\alpha \mid \alpha < \kappa)$ recursively and satisfying $M_{\alpha+1} \in \mathcal S$ contains $m_\alpha$ for every $\alpha<\kappa$. Set $M_0=0$. If $M_\alpha$ has been constructed for some $\alpha < \kappa$, take $M_{\alpha+1} \in \mathcal S$ such that $M_\alpha+m_{\alpha}R \leq M_{\alpha+1}$. 
By Remark~\ref{l:DirectSummand},  $M_{\alpha+1}\leq_{\oplus}M_{\beta+1}$ for $\alpha < \beta < \kappa$. In particular, $M_{\alpha+1}$ is a direct summmand of $M_\beta$ as well.
\end{proof}

In order to see if a strongly $(\kappa,\Add(X))$-free module is trivial, we can look at another subset of $\kappa$ different from the one in Lemma \ref{WhenMistrivial2}

\begin{proposition}\label{p:Strongly free are trivial}
    Let $\kappa$ be an uncountable regular cardinal and $X$ a module with $\lambda_X < \kappa$. The following are equivalent for a strongly $(\kappa,\Add(X))$-free and $\kappa$-generated module $M$:
    \begin{enumerate}
        \item $M$ is trivial;

        \item there exists a $\kappa$-filtration of $M$, $\langle M_\alpha \mid \alpha < \kappa \rangle$, with $M_{\alpha+1} \in \Add_\kappa(X)$ and $M_{\alpha+1} \leq_\oplus M_\beta$ for each $\alpha < \beta < \kappa$, such that the set
        \begin{displaymath}
            E''=\{\alpha < \kappa \mid M_\alpha \nleq_\oplus M_{\alpha+1}\}
        \end{displaymath}
        is not stationary in $\kappa$;

        \item for every $\kappa$-filtration of $M$, $\langle M_\alpha \mid \alpha < \kappa \rangle$, with $M_{\alpha+1} \in \Add_\kappa(X)$ and $M_{\alpha+1} \leq_\oplus M_\beta$ for each $\alpha < \beta < \kappa$, that the set
        \begin{displaymath}
            E''=\{\alpha < \kappa \mid M_\alpha \nleq_\oplus M_{\alpha+1}\}
        \end{displaymath}
        is not stationary in $\kappa$.
    \end{enumerate}
\end{proposition}

\begin{proof}
    Simply notice that every $\kappa$-filtration $\langle M_\alpha \mid \alpha < \kappa\rangle$ of $M$ as in Lemma \ref{l:StrongFree} we have the equality
    \begin{displaymath}
        E''=\{\alpha < \kappa \mid \{\beta > \alpha \mid M_\alpha \nleq_\oplus M_\beta \}\textrm{ is stationary in }\kappa\}
    \end{displaymath}
    Then the result follows from Lemma \ref{WhenMistrivial2}.
\end{proof}

\section{On $(\kappa,\Add(X))$-separable modules}\label{sec: separable}
Continuing with our study of the class of (strongly) $(\kappa,\mathrm{Add}(X))$-free modules in this section we shall be preoccupied with the notion of $(\kappa,\mathrm{Add}(X))$-separability:
\begin{definition}\label{def: separable}
  % Let $X$ and $\kappa$ be as above. 
Let $\kappa$  be an infinite regular cardinal and $\mathcal X$, a class of modules.  A module $M$ is called  \emph{$(\kappa, \mathcal X)$-separable} if for each $Y \subseteq M$ with $|Y|<\kappa$ there is $N\leq_{\oplus} M$ with $N\in \mathcal X$ and $Y\s N$.
  
  A $(\kappa, \mathcal X)$-separable module $M$ will be called \emph{trivial} whenever $M\in \mathcal X$.
\end{definition}

In this paper we are interested in $(\kappa,\Add(X))$ and $(\kappa,\Add_\kappa(X))$-separable modules for some fixed module $X$. This seemingly strengthening of $(\kappa,\mathrm{Add}(X))$-separability is equivalent to the former whenever $X$ is a ${<}\kappa$-generated module. Proof of this will be provided  in Lemma~\ref{l:equivalencebetweenseparability} below.

\begin{example}
To demystify  the concept of $(\kappa, \mathrm{Add}_\kappa(X))$-separability  let us consider a concrete example -- the case where $\kappa=\aleph_0$ and $X$ is the ring, $R$. In this scenario a module $M$ is $(\aleph_0,\mathrm{Add}_{\aleph_0}(R))$-separable whenever every finite set $F\s M$ is contained in a finitely-generated projective submodule $N\leq_{\oplus} M.$ Namely, $M$ is a separable module in the traditional sense.
\end{example}
%A trivial observation is that every $M\in \mathrm{Add}(X)$ is $(\kappa, \mathrm{a}$
%The no $(\kappa, \mathrm{Add}(X))$-separability is only of interest when $\kappa$ is a regular cardinal:
%\begin{question}
 % Suppose that $\kappa$ is a singular cardinal. Is every $(\kappa, \mathrm{Add}(X))$-separable module trivial?
%\end{question}
%\begin{proof}
 %   \textcolor{blue}{TODO. If true, it should follow from Shelah's compactness theorem. However, this does not seem to be a compactness-type property. Every small subset of $M$ is contained in a member of $\mathrm{Add}(X)$ but it is false that every small subset of $M$ belongs to $\mathrm{Add}(X)$. }
%\end{proof}
%\man{Usar en todo el artículo una de dos: $\lambda_X<\kappa$ o $X$ es  $<\kappa$-generado}\ale{I think it's fine to pivot between both terminologies. I think it's more clear to say $X$ is ${<}\kappa$-generated but often times the notation $\lambda_X<\kappa$ becomes handy.}
We should like to begin clarifying the connection between $(\kappa,\Add(X))$-separability and strongly $(\kappa,\Add(X))$-freeness.  Namely,

\begin{lemma}\label{lemma: separable implies strongly free}
    Let $\kappa$ be an uncountable regular cardinal and $X$, a module with $\lambda_X<\kappa$. If $M$ is $(\kappa,\Add(X))$-separable then it is strongly $(\kappa,\Add_\kappa(X))$-free.
\end{lemma}

\begin{proof}
    Let $\mathcal{S}$ be the set of all $<\kappa$-generated direct summands of $M$ that belong to $\Add(X)$. Given $S\in \mathcal{S}$ and $Y\s M$ with $|Y|<\kappa$, we can find, by the $(\kappa,\Add(X))$-separability, $S'\leq_{\oplus} M$ in $\Add(X)$ such that $Y\cup S\s S'$. Since $S'$ is a direct sum of $<\kappa$-generated modules by Walker's lemma, we can find a $<\kappa$-generated direct summand $S''$ of $S'$ containing $Y \cup S$. Then, $S''$ is a direct summand of $M$, so that it belongs to $\mathcal S$, and $S$, being a direct summand of $M$, is a direct summand of $S''$ as well.
\end{proof}

\begin{remark}
Notice that $(\kappa,\Add(X))$-separable modules  need not be $(\kappa,\Add(X))$-free, since strongly $(\kappa,\Add(X))$-free modules need not be $(\kappa,\Add(X))$-free by \cite[Theorem 8]{Trlifaj}.
\end{remark}

Now we give the characterization of separable modules in terms of filtrations:

\begin{lemma}\label{l:equivalencebetweenseparability}
    Let $\kappa$ be an infinite regular cardinal. The following  are equivalent for a $\kappa$-generated module $M$; namely,
    \begin{enumerate}
        \item $M$ is $(\kappa,\mathrm{Add}_\kappa(X))$-separable
        \item $M$ is $(\kappa,\mathrm{Add}(X))$-separable;
        \item $M$ admits a $\kappa$-filtration, $\langle M_\alpha \mid \alpha < \kappa \rangle$ with $M_{\alpha+1} \in \Add_\kappa(M)$ and $M_{\alpha+1} \leq_\oplus M$ for each $\alpha < \kappa$.
    \end{enumerate}
\end{lemma}
\begin{proof}

 (1) $\Rightarrow$ (2): Is obvious.
 
 \smallskip

    (2) $\Rightarrow$ (3): The proof of Lemma \ref{lemma: separable implies strongly free} implies that there exists a set $\mathcal S$ of direct summands of $M$ witnessing the strong $(\kappa,\Add_\kappa(X))$-freeness of $M$. Using the proof of (2) $\Rightarrow$ (3) of Lemma \ref{l:StrongFree}, we can find a $\kappa$-filtration of $M$, $\langle M_\alpha \mid \alpha < \kappa\rangle$, with $M_{\alpha+1} \in \mathcal S$ for each $\alpha < \kappa$. In particular, $M_{\alpha+1}\leq_{\oplus} M$ and $M_{\alpha+1}\in \Add_\kappa(X)$ for each $\alpha < \kappa$.

(3) $\Rightarrow$ (1): %Suppose that $\langle M_\alpha\mid \alpha<\kappa\rangle$ is a $(\kappa, \mathrm{Add}(X))$-filtration.  
Trivial.
\end{proof}
%One consequence of this result is that we can use Proposition \ref{p:Strongly free are trivial} to check if a $(\kappa,\mathrm{Add}(X))$-separable module is trivial. \ale{I do not understand this sentence. Propo 3.10 refers to when a free group is trivial.}\smallskip

In the next section we will employ Lemma~\ref{l:equivalencebetweenseparability} to produce a $\kappa^+$-generated non-trivial $(\kappa^+,\mathrm{Add}(X))$-separable module for an infinite (non-necessarily regular)  $\kappa$ -- this will yield an extension of  \cite[Theorem~4.7]{CortesGuilTorrecillas}. In particular, this will provide an example of a non-trivial strongly $(\kappa^+,\Add(X))$-free module (see Lemma~\ref{lemma: separable implies strongly free}). %\ale{Maybe this sentence has to be redacted? }The main difference between the module we construct next with the one of Theorem~\ref{t:ConstructionStronglyFree} is that the former will not be $(\kappa^+,\Add(X))$-free.

\subsection{A construction of  a non-trivial $(\kappa,\mathrm{Add}(X))$-separable module}

The main construction of this section is based on the notions of a template and of a tree-like ladder system, both introduced next:

\begin{definition}\label{template}
    Let $\lambda$ be an infinite regular cardinal and $X$, a module. A triple $\langle N, L, {\mathcal{N}}\rangle$ is a \emph{$(\lambda,\Add(X))$-template} if $N\s L$ are modules, $\mathcal{N}$ is a continuous filtration of $N$, $\langle N_\alpha\mid \alpha<\lambda\rangle$, satisfying $N_\alpha\leq_{\oplus} L$, $L \in \Add(X)$ and $N\nleq_{\oplus} L.$
\end{definition}

%The next concept will be  instrumental in Section~\ref{sec: separable}:
\begin{definition}[Tree-like ladder system]
 Let $\kappa$ be an uncountable regular cardinal and $S\s\kappa$, a stationary set consisting of limit ordinals of some fixed cofinality $\lambda < \kappa$. A sequence $\langle c_\eta\mid \eta\in S\rangle$ is called a \emph{ladder system} if $c_\eta\colon \lambda\rightarrow \eta$ is a strictly increasing function and the range of $c_\eta$ is unbounded in $\eta.$ A ladder system $\langle c_\eta\mid \eta\in S\rangle$  is said to be \emph{tree-like} if for each $\eta, \zeta\in S$,  $\delta,\delta'<\lambda$  $$\text{$c_\eta(\delta)=c_\zeta(\delta')\;\Rightarrow\; \delta=\delta'\;\text{and}\; c_\eta\restriction \delta=c_\zeta\restriction\delta.$}$$
\end{definition}

 \begin{remark}\label{RemarkOnLadders}
 Given a tree-like ladder system $\langle c_\eta\mid \eta\in S\rangle$  one can define another tree-like ladder system by stipulating  $c^*_\eta(\xi)=c_\eta(\xi)+1$ for any $\xi<\lambda$. As a consequence, we may assume that $\mathrm{Im}(c_\eta)\cap S=\emptyset$ and that $\Img(c_\eta)\cap \{s+1\mid s \in S\}=\emptyset$ for all $\eta\in S$.
We will assume that all of our ladder systems  satisfy this.
 \end{remark}

Every stationary subset $S$ (consisting of ordinals of some fixed cofinality) of every uncountable cardinal $\kappa$ carries a ladder system. If $\kappa =\aleph_1$, this ladder system can be constructed with the tree-like ladder property (see \cite[Exercise XII.17]{EklofMekler}). If $\kappa>\aleph_1$ and $S$ belongs to the approchability ideal of $\kappa$, then it is possible to find a club $C$ such that $E \cap C$ has a tree-like ladder system \cite[Lemma VI.5.13]{EklofMekler}. Other tree-like ladder systems have been constructed using aditional set theoretical hypothesis, see \cite[Theorem 9]{Eklof2} and \cite[Proposition 2.1]{CortesGuilTorrecillas}.

\begin{definition}
    Given cardinals  $\lambda<\mu$ with $\lambda$ regular %and a regular cardinal $\lambda$ with  $\lambda<\kappa$ we 
denote by $(\star)_{\mu,\lambda}$ the conjunction of the next two sentences: % the next two assertions:
\begin{enumerate}
	\item There is a non-reflecting stationary set $S\s E^{\mu^+}_{\lambda}$.
	\item There is a tree-like ladder system
 $\langle c_\eta\mid \eta\in S\rangle$. 
 %consists either of successor ordinals 
% \begin{itemize}
 %    \item The range of $c_\eta$ consists either of successor ordinals 
% \end{itemize}
 \end{enumerate}
\end{definition}

 \begin{remark}
If $\mu$ is a singular cardinal and $\lambda<\mu$ is regular, the existence of a stationary set $S\s E^{\mu^+}_{\lambda}$  in the approachability ideal $I[\mu^+]$ turns to be  a ZFC theorem -- this is due to Shelah (see \cite[Theorem~3.18]{Eisworth})). Combining this with \cite[Lemma~5.13]{EklofMekler} one can construct (in ZFC) a tree-like ladder system  supported on an  approachable stationary $S\s E^{\mu^+}_\lambda$.
However, it is consistent with ZFC that every stationary subset of $\mu^+$ reflects. This holds  in  Magidor's model for stationary reflection  at $\aleph_{\omega+1}$ \cite{Magidor}. For details see \cite[Corollary~3.41]{Eisworth}.
\end{remark}
 Under appropriate set-theoretic assumptions upon $\mu$ one can prove that  $(\star)_{\mu,\lambda}$ is consistent with the ZFC axioms. For instance:
 \begin{lemma}\label{l:StarIsConsistent}
 	If $\square_\mu$ holds then $(\star)_{\mu,\lambda}$ holds for all $\lambda<\mu$ regular.
 	
 	In particular, if $0^\sharp$ does not exist then $(\star)_{\mu,\lambda}$ holds for all singular cardinals $\mu$ and $\lambda<\mu$ regular.%\man{En la segunda parte, el principio se verifica para todo $\mu$ singular, no?}
 \end{lemma}
 \begin{proof}
 	Since $\square_\mu$ holds we can let a non-reflecting stationary set $S\s E^{\mu^+}_{\lambda}$ \cite[Theorem~2.1]{CumForMag}. In addition $S$ is approachable (i.e., $S\in I[\mu^+]$) because so is $\mu^+$ and $I[\mu^+]$ is an ideal. Now \cite[Lemma~5.13]{EklofMekler} yields a club $C\s\mu^+$ and a tree-like ladder system $\langle c_\eta\mid \eta\in S^*\rangle$ where $S^*:=S\cap C$. By \cite[Theorem 2.4]{Eisworth}, $S^*$ is a non-reflecting stationary subset of $E^{\mu^+}_{\lambda}$ and $\langle c_\eta\mid \eta\in S^*\rangle$ is the sought ladder system. The last assertions follows from the fact that if $0^\sharp$ does not exists then $\square_\mu$ holds for every singular cardinal $\mu$.
 \end{proof} 
We remind our readers that $``\lambda_X<\kappa$'' was a shorthand for $``X$ is ${<}\kappa$-generated''. The main theorem of the section reads as follows:
\begin{theorem}\label{MainTheorem}
   Let $\lambda=\cf(\lambda)<\kappa$ be cardinals  witnessing $(\star)_{\kappa,\lambda}$. Let  $X$ be a module with $\lambda_X < \kappa$, and suppose that there is a $(\lambda,\mathrm{Add}(X))$-template, $\langle N,L,\mathcal{N}\rangle$.  Then there is a $\kappa^+$-generated non-trivial $(\kappa^+,\Add(X))$-free and $(\kappa^+,\mathrm{Add}(X))$-separable module.
\end{theorem}

Let  $\langle c_\eta\mid \eta\in S\rangle$ be a tree-like ladder system   witnessing $(\star)_{\kappa,\lambda}$. %supported on a non-reflecting stationary set $S\s E^{\kappa^+}_{\lambda}$.
The module $M$ witnessing the thesis of Theorem~\ref{MainTheorem} will be obtained as a union of a $\kappa^+$-filtration, $\langle M_\alpha\mid \alpha<\kappa^+\rangle$, such that $M_{\alpha +1} \in \Add(X)$ for each $\alpha < \kappa^+$ and $$S\s \{\alpha<\kappa^+\mid M_{\alpha}\nleq_{\oplus} M_{\alpha+1}\}.$$ This will entail the non-triviality of $M$ (i.e., $M\notin\Add(X)$) by virtue of Proposition \ref{p:Strongly free are trivial} and Lemma~\ref{l:equivalencebetweenseparability}.  % is non-trivial.

\smallskip

The idea is to construct each  $M_{\alpha+1}$ by taking direct sums of isomorphic copies of members $N_\nu$ of our fixed template. Before entering into further considerations let us agree upon some notations.
\begin{notation}
    For each $\nu<\lambda$ and $\gamma<\kappa^+$ let us fix $N_{\gamma,\nu}$ an isomorphic copy of $N_\nu$. Also, since $N_\nu\leq_{\oplus}N_{\nu+1}$  there is $N'_\nu$ such that $N_{\nu+1}=N_\nu\oplus N'_\nu$. Thus, for each $\gamma<\kappa^+$, we may let isomorphic copies $N^*_{\gamma,\nu}\simeq N_\nu$ and $N'_{\gamma,\nu}\simeq N'_\nu$ such that $$N_{\gamma,\nu+1}=N^*_{\gamma,\nu}\oplus N'_{\gamma,\nu}.$$
    Let us fix  an isomorphism $s^\nu_{\gamma}\colon N'_\nu\rightarrow N'_{\gamma,\nu}$. 
\end{notation}

Let us proceed with the construction. First, set $M_0:=0$. Assuming that $\langle M_\beta\mid \beta<\alpha\rangle$ has been defined we construct $M_\alpha$ as follows. If $\alpha$ is a limit ordinal we simply set $M_{\alpha}:=\bigcup_{\beta<\alpha} M_\beta$. If $\alpha=\alpha_*+1$  we distinguish among two cases:

\medskip

\underline{\textbf{Case $\alpha_*\notin S$}:}\label{EasyCase} Let $M'_{\alpha_*}=\oplus_{\nu<\lambda} N_{\alpha_*,\nu}$ and define
$$M_{\alpha}:=M_{\alpha_*}\oplus M'_{\alpha_*}.$$ 

\smallskip

\underline{\textbf{Case $\alpha_*\in S$}:} In this case $M_{\alpha}$ is defined as the pushout between the inclusion $i\colon N\rightarrow L$ and a morphism $\iota_{\alpha_*}\colon N\rightarrow M_{\alpha_*}$ that we are yet to define. In turn, $\iota_{\alpha_*}$ will be defined as the limit of a direct system of morphisms $$\langle \iota^\nu_{\alpha_*}\colon N_\nu\rightarrow M_{\alpha_*},\; \nu<\lambda\rangle.$$
 First, $\iota^0_{\alpha_*}$ is declared to be the zero isomorphism. If we have constructed the direct system $\langle \iota^\sigma_{\alpha_*}\mid \sigma<\nu\rangle$ for some $\nu < \lambda$, then one takes $\iota^\nu_{\alpha_*}:=\varinjlim_{\sigma<\nu} \iota^\sigma_{\alpha_*}$ whenever $\nu$ is a limit ordinal; otherwise, if $\nu=\sigma_*+1$ one takes  $\iota^\nu_{\alpha_*}:=\iota^{\sigma_*}_{\alpha_*}\oplus s^{\sigma_*}_{c_{\alpha_*}(\sigma_*)}$. We note that $\iota^\nu_{\alpha_*}$ is well-defined. In the limit case this is evident and in the successor case it follows from
\begin{itemize}
	\item $s^{\sigma_*}_{c_{\alpha_*}(\sigma_*)}\colon N'_{\sigma_*}\rightarrow N'_{c_{\alpha_*}(\sigma_*), \sigma_*}$
	\item and $ N'_{c_{\alpha_*}(\sigma_*), \sigma_*} \leq_{\oplus} M'_{c_{\alpha_*}(\sigma_*)}\leq M_\alpha$ (as $c_{\alpha_*}(\sigma_*)<\alpha_*$).
\end{itemize}   
Finally, put $\iota_{\alpha_*}:=\varinjlim_{\nu<\lambda} \iota^\nu_{\alpha_*}$.\label{pushout} %Finally, as mentioned above, 
Let $M_{\alpha_*
+1}$ be the outcome of the pushout %the inclusion $i\colon N\rightarrow L$ and $\iota_{\alpha_*}\colon N\rightarrow M_{\alpha}$. This yields a commutative diagram
\begin{equation}\label{e:Pushout}
\begin{tikzcd}
N \arrow{r}{i} \arrow{d}{\iota_{\alpha_*}} & L \arrow{d}{{\theta_{\alpha_*}}} \\
  M_{\alpha_*} \arrow{r}{i_{\alpha_*}} & M_{\alpha_*+1}.
\end{tikzcd}
\end{equation}
%Let $i_{\alpha_*}\colon M_{\alpha_*}\rightarrow M_{\alpha_*+1}$ be the morphism arising from this pushout.

\medskip

The above completes our construction of the sequence $\langle M_\alpha\mid \alpha<\kappa^+\rangle$. Let us now prove that this is indeed a $\kappa^+$-filtration satisfying Lemma \ref{l:equivalencebetweenseparability}.

First, we see that $M_{\alpha}$ is a direct summand of $M$ when $\alpha$ does not belong to $S$. We will use the following technical fact:

\begin{lemma}\label{l:Technical}
    Let $\eta \in S$ and $\delta < \kappa^+$ with $\delta < \eta$. Set $\tau_{\eta,\delta}=\sup\{\sigma+1\mid c_\eta(\sigma)\leq \delta\}$. Then, for any $\nu < \tau_{\eta,\delta}$, $c_\eta(\nu)\leq \delta$.
\end{lemma}

\begin{proof}
    If $c_\eta(\nu)>\delta$, then $\nu$ is an upper bound of the set $\{\sigma+1\mid c_\eta(\sigma)\leq \delta\}$ so that $\tau_{\eta,\delta}\leq \nu$. Consequently, if $\nu < \tau_{\eta,\delta}$ then it is satisfied that $c_\eta(\nu) \leq \delta$.
\end{proof}

\begin{lemma}\label{LemmaDirectSummads}
     For each $\alpha\notin S$, $M_{\alpha}\leq_{\oplus} M$.

     In particular, $M_{\alpha+1}\leq_{\oplus} M$ for all $\alpha<\kappa^+.$
 \end{lemma}
 \begin{proof}
  %The construction is the same as the one provided in \cite[Proposition~4.4]{CortesGuilTorrecillas} but we spell out the details for the reader's convenience. 
  Fix $\alpha\notin S$. One has to consider two cases; namely, either $\alpha$ is  successor  or it is limit. The latter case will  follow from the former in that $M_{\alpha}\leq_{\oplus} M_{\alpha+1}\leq_{\oplus} M$. Consequently we focus on analyzing the case where $\alpha$ is of the form $\delta+1$.

  \smallskip

  First, for each $\nu<\lambda$ (as $N_\nu\leq_{\oplus} L$) we fix $\pi_{\nu}\colon L\rightarrow L$  an idempotent endomorphism such that $\mathrm{Im}(\pi_\nu)=N_\nu$. For each $\eta\in S$ above $\delta+1$ let $\tau_\eta$ the ordinal $\tau_{\eta,\delta}$ defined in Lemma \ref{l:Technical}.

  \smallskip

To show that $M_{\delta+1}$ is a direct summand of $M$ we construct a projection $$p\colon M\rightarrow M_{\delta+1}.$$ Namely, $p$ will be a homomorphism such that $p\restriction M_{\delta+1}=1_{M_{\delta+1}}.$ In turn, $p$ is defined as the direct limit of a direct system of homomorphisms $$\langle p_\beta\colon M_\beta\rightarrow M_{\delta+1}\mid \beta<\kappa^+\rangle$$ satisfying:
\begin{enumerate}
    \item[(P1)] If $\beta \leq \delta+1$, then $p_\beta \rest M_{\delta+1}=1_{M_{\delta+1}}$, and

    \item[(P2)] if $\beta=c_\eta(\sigma)>\delta+1$ for some $\eta>\delta$ belonging to $S$, then $p_{\beta+1}\rest N'_{\beta,\sigma}=\iota_\eta\circ\pi_{\tau_\eta}\circ(s_\beta^\sigma)^{-1}$.
\end{enumerate}

We construct the said system by induction on $\beta<\kappa^+$. For $\beta\leq \delta+1$ take simply $p_\beta$ be the inclusion map. Suppose that $\langle p_{\beta'}\mid \beta'<\beta\rangle$ has been defined. If $\beta$ happens to be a limit ordinal we take $p_\beta$ the direct limit of the previous $p_{\beta'}$'s. Otherwise, $\beta$ takes the form $\mu+1$ and we have to do something clever. We now distinguish three cases:

  \smallskip

  \underline{\textbf{Case $\mu\notin S$ and $\mu\notin \bigcup_{\eta\in S}\mathrm{Im}(c_\eta)$:}} In this case we let $p_{\mu+1}:=p_{\mu}\oplus 0.$

  \medskip

    \underline{\textbf{Case $\mu\notin S$ and $\mu\in \bigcup_{\eta\in S}\mathrm{Im}(c_\eta)$:}} In this case $M_{\mu+1}$ decomposes as
    $$M_\mu \oplus M'_\mu = M_{\mu}\oplus (\bigoplus_{\nu\neq \sigma+1}N_{\mu, \nu})\oplus N^*_{\mu,\sigma}\oplus N'_{\mu,\sigma}$$
    where $\sigma$ is such that $\mu=c_\eta(\sigma)$ for some $\eta\in S$. Consider the morphism $q_\mu=\iota_\eta\circ\pi_{\tau_\eta}\circ(s_\beta^\sigma)^{-1}$ defined on $N'_{\mu,\sigma}$ whose image, by Lemma \ref{l:Technical}, is contained in $M_{\delta+1}$. Thus, $p_{\mu+1}:=(p_\mu\oplus 0\oplus 0 \oplus q_\mu)$ defines a homomorphism between $M_{\mu+1}$ and $M_{\delta+1}$ which trivially satisfies (P2) above.

    We are left to show that this homomorphism does not depend upon the choice of $\eta\in S$ -- here is where the tree-likeness of $\langle c_\eta\mid \eta\in S\rangle$ will come into play. Suppose that $c_\eta(\sigma)=\mu=c_\xi(\rho)$. By tree-likeness, $\rho=\sigma$ (in particular, $s^\sigma_\mu=s^\xi_\mu$) and $$c_\eta\restriction \sigma+1=c_\xi\restriction\rho+1.$$  Now, since $c_\eta(\sigma)=c_\xi(\rho)>\delta$ it follows that $\tau_\eta=\tau_\rho$ and thus $\pi_{\tau_{\eta}}=\pi_{\tau_\rho}$. Finally, observe that $\iota_\eta\restriction N_{\tau_\eta}=\iota_\xi\restriction N_{\tau_\xi}$ as both homomorphism were constructed using the same $s^\sigma_\beta$'s in that $c_\eta\restriction \sigma=c_\xi\restriction\sigma$.

\medskip

     \underline{\textbf{Case $\mu\in S$:}} %In this case $M_{\delta+1}$ arises from the pushout  construction. 
  If we consider the morphism $\theta_\mu \pi_\mu$ from $L$ to $M_{\mu+1}$, whose image is inside $M_{\delta+1}$ by Lemma \ref{l:Technical} we get, by (P2), the following commutative diagram
  \begin{displaymath}
\begin{tikzcd}
 N \arrow{r}{i} \arrow{d}{\iota_\mu} & L \arrow{d}{\theta_\mu\pi_{\tau_\mu}} \\
 M_\mu \arrow{r}{p_\mu} & M_{\delta+1}.
\end{tikzcd}
\end{displaymath}
Since $M_{\mu+1}$ is the pushout of $\iota_\mu$ and $i$, the universality of this construction yields a homomorphism $p_{\mu+1}\colon M_{\mu+1}\rightarrow M_{\delta+1}$ such that $p_{\mu+1}\circ\ i_{\mu}=p_\mu$ and $p_{\mu+1}\theta_\mu=\theta_\mu\pi_{\tau_\mu}.$
Clearly, $p_{\mu+1}$ is as desired.
 \end{proof}

The next step is to prove that $M$ is not trivial, equivalently, by Proposition \ref{p:Strongly free are trivial} and Lemma \ref{lemma: separable implies strongly free}, that $\{\alpha<\kappa^+\mid M_{\alpha}\nleq_{\oplus} M_{\alpha+1}\}$ is stationary. We need the following result about $\iota_\mu$.

\begin{lemma}\label{l:i_mu}
    The morphism $\iota_\mu$ is a split monomorphism for every $\mu \in S$.
\end{lemma}

 \begin{proof}
    First, observe that $$\textstyle\mathrm{Im}(\iota_{\mu})=\bigoplus_{\nu<\lambda} N'_{c_{\mu}(\nu),\nu}$$ and that $N'_{c_{\mu}(\nu),\nu}\leq_{\oplus} N_{c_{\mu}(\nu),\nu+1}\leq_{\oplus} M'_{c_{\mu}(\nu)}.$ Thus, $$\textstyle\mathrm{Im}(\iota_{\mu})\leq_{\oplus}\bigoplus_{\nu<\lambda} M'_{c_{\mu}(\nu)}.$$

     Let us show that $\bigoplus_{\nu<\lambda} M'_{c_{\mu}(\nu)}$ is a direct summand of $M_{\mu}$. To show this let 
     $$C:=\mathrm{Im}(c_{\mu})\cup\{\beta+1\mid \beta\in \mathrm{Im}(c_{\mu})\}\cup \{\sup\{c_\mu(\nu)\mid \nu < \sigma\} \mid \sigma < \lambda \text{ limit}\}.$$

     \begin{claim}
         $C$ is a club in $\mu$ with $C \cap S = \emptyset$.
     \end{claim}     

     \begin{proof}[Proof of claim]
         Clearly, $C \cap S=\emptyset$ because if $\alpha \in C$ is successor, then $\alpha \notin S$ by Remark \ref{RemarkOnLadders}, and if $\alpha$ is limit, then its cofinality is smaller than $\lambda$, and the ordinals in $S$ have cofinality equal to $\lambda$.

         Since $C$ is unbounded in $\mu$, it remains to see that it is closed. Let $A \subseteq C$ be a subset of $C$ with $\sup A < \mu$. If $A$ has a cofinal subset $A'$ consisting of ordinals belonging to $\mathrm{Im}(c_{\mu})\cup\{\beta+1\mid \beta\in \mathrm{Im}(c_{\mu})\}$, then $\sup A = \sup A' \in \{\sup\{c_\mu(\nu)\mid \nu < \sigma\} \mid \sigma < \lambda \text{ limit}\} \subseteq C$. Otherwise, $A$ has a cofinal subset $A'$ consisting of ordinals belonging to $\{\sup\{c_\mu(\nu)\mid \nu < \sigma\} \mid \sigma < \lambda \text{ limit}\}$. Take $\beta < \lambda$ such that $A'=\{\alpha_\gamma \mid \gamma < \beta\}$ and, for each $\gamma < \beta$, take $\sigma_\gamma < \lambda$ satisfying

         $$\alpha_\gamma = \sup \{c_\mu(\nu) \mid \nu < \sigma_\gamma\}.$$

         Now let $\sigma=\sup \{\sigma_\gamma \mid \gamma < \beta\}$ and notice that, since $\sup A'<\mu$, $\sigma < \lambda$. Then it is easy to see that

         $$\sup A = \sup A' = \sup \{c_\mu(\nu)\mid \nu < \sigma\}$$

         and, consequently, $\sup A \in C$ also. This finishes the proof of the claim.
     \end{proof}
     
     Let $C:\lambda \rightarrow C$ be a continuous strictly increasing map and define de filtration $\{Q_\sigma \mid \sigma < \lambda\}$ of $M_\mu$ as declaring $Q_\sigma = M_{C(\sigma)}$. Since $C \cap S\neq \emptyset$, $Q_\sigma$ is a direct summand of $M_\mu$ for every $\sigma < \lambda$ by Lemma \ref{LemmaDirectSummads}. Now, for $\sigma < \lambda$, notice the following:
     \begin{itemize}
         \item If $C(\sigma)=c_\mu(\nu)$ for some $\nu < \lambda$, then 
         \begin{displaymath}
             Q_{\sigma+1}=M_{c_\mu(\nu)+1}=Q_\sigma \oplus M'_{c_\mu(\nu)}.
         \end{displaymath}

         \item If $C(\sigma)\neq c_\mu(\nu)$ for every $\nu < \lambda$, there exists a submodule $B_\sigma$ of $M$ such that $Q_{\sigma+1}=Q_\sigma \oplus B_\sigma$.
     \end{itemize}
     By Fact \ref{fact: when M is a direct sum},
     \begin{displaymath}
         M=\left(\bigoplus_{\sigma < \lambda}M'_{c_\mu(\sigma)}\right)\bigoplus\left(\bigoplus_{\substack{\sigma < \lambda\\C(\sigma)\notin \Img(c_\mu)}}B_\sigma\right),
     \end{displaymath}
  which concludes the proof.
 \end{proof}
 
\begin{lemma}
$S\s \{\alpha<\kappa^+\mid M_{\alpha}\nleq_{\oplus} M_{\alpha+1}\}$.
\end{lemma}
\begin{proof}
	Towards a contradiction suppose that $M_{\alpha_*}\leq_{\oplus} M_{\alpha_*+1}$ for some ${\alpha_*}\in S$. Then $i_{\alpha_*}$ is a split monomorphism so that $\theta_{\alpha_*}i=i_{\alpha_*}\iota_{\alpha_*}$ is a split monomorphism by the previous lemma. Then $i$ is a split monomorphism as well, which is a contradiction since $\langle N,L,\mathcal N\rangle$ is a template.
\end{proof}

The last part of the proof is the fact that every module in the filtration belongs to $\Add(X)$. Here, we use that $S$ does not reflect.

\begin{lemma}
	$M_{\alpha}\in \mathrm{Add}(X)$ for all $\alpha\notin S$. 

 In particular, $M_{\alpha+1}\in\Add(X)$ for all $\alpha<\kappa^+$.
\end{lemma}
\begin{proof}
Let $\alpha\notin S$ and suppose by induction that $M_\beta\in\Add(X)$ for all $\beta<\alpha$ not in $S$. If $\alpha=\beta_*+1$ then we distinguish two cases: either $\beta_*\notin S$ or $\beta_*\in S$. In the former case, $M_\alpha:=M_{\beta_*}\oplus \bigoplus_{\nu<\lambda} N_{\beta_*,\nu}$, a direct sum of members of $\Add(X)$ and thus a member of $\Add(X).$ Alternatively, $\beta_*\in S$ and $M_\alpha$ is the outcome of the pushout diagram
  \begin{displaymath}
\begin{tikzcd}
N \arrow{r}{i} \arrow{d}{\iota_{\beta_*}} & L \arrow{d}{\theta_{\beta_*}} \\
 M_{\beta_*} \arrow{r}{i_{\beta_*}} & M_{\alpha}.
\end{tikzcd}
\end{displaymath}
By Lemma~\ref{l:i_mu}, $\iota_{\beta_*}$ is a split monomorphism hence so is $\theta_{\beta_*}$. Using that $M_{\alpha}\cong \Coker(\theta_{\beta_*})\oplus L$ and that $\Coker(\theta_{\beta_*})=\Coker(\iota_{\beta_*}) \in \Add(X)$ by induction hypothesis, we get that $M_\alpha$ belongs to $\Add(X)$ as well.

Finally, suppose that $\alpha$ is a limit ordinal. If $\cf(\alpha)=\omega$ we fix an increasing cofinal sequence $\langle \beta_n+1\mid n<\omega\rangle$ converging to $\alpha$. Since $M_{\beta_{n+1}+1}=M_{\beta_n+1}\oplus M'_{\beta_{n+1}+1}$ (see Lemma~\ref{LemmaDirectSummads}) Fact~\ref{fact: when M is a direct sum} yields $M_\alpha=\bigoplus_{n<\omega} M_{\beta_{n+1}+1}'$ and as a result $M_\alpha\in \Add(X).$ In case $\cf(\alpha)\geq \omega_1$  we use the fact that $S\cap \alpha$ is non-stationary to find a club $C\s \alpha$ avoiding $S$ and argue as before with $\langle M_\beta\mid \beta\in C\rangle$.
%	Suppose by induction  that $M_{\beta+1}\in \mathrm{Add}(X)$ for all $\beta<\alpha$. First, we assume that $\alpha$ is of the form $\alpha_*+1$ for some ordinal  $\alpha_*\notin S$. In that case, by definition, $$\textstyle M_\alpha:=M_{\alpha_*}\oplus (\bigoplus_{\nu<\lambda} N_{\alpha_*,\nu}).$$
	%If $\alpha_*$ is a successor ordinal then the induction hypothesis combined with the fact that $N_{\alpha_*,\nu}\simeq N_\nu\in \mathrm{Add}(X)$ yields $M_{\alpha}\in\mathrm{Add}(X)$. If $\alpha_*$ is a limit ordinal then standard arguments yield $M_{\alpha_*}=\bigoplus_{\beta<\alpha_*}(\bigoplus_{\nu<\lambda} N_{\beta,\nu})\in\mathrm{Add}(X)$. So we are left with discussing the case where $\alpha=\alpha_*+1$ and $\alpha_*\in S$. Since $\iota_{\alpha_*}\colon N\rightarrow M_{\alpha_*}$ is a split monomorphism then so is $\theta_\alpha\colon L\rightarrow M_{\alpha_*+1}$ and hence $M_{\alpha_*+1}=L\oplus \mathrm{Ker}(\iota'_{\alpha_*})$ where $\iota'_{\alpha_*}\colon M_{\alpha_*}\rightarrow N$ witnesses $\iota'_{\alpha_*}\circ \iota_{\alpha_*}=\mathrm{id}_{N}$. Since $M_{\alpha_*}\in\mathrm{Add}(X)$ and $\mathrm{Ker}(\iota'_{\alpha_*})\leq_{\oplus} M_{\alpha_*}$ it follows that $\mathrm{Ker}(\iota'_{\alpha_*})\in\mathrm{Add}(X)$. Since $L\in\mathrm{Add}(X)$ it follows that $M_{\alpha_*+1}\in\Add(X)$, as needed.
	\end{proof}

 \begin{lemma}
	$M_\alpha$ is ${\leq}\kappa$-generated for all $\alpha<\kappa^+$. Actually:
 \begin{enumerate}
     \item If $\alpha < \max\{\lambda,\lambda_X\}$, then $M_\alpha$ is $\leq \max\{\lambda,\lambda_X\}$-generated.

     \item If $\max\{\lambda,\lambda_X\} \leq \alpha < \kappa^+$, then $M_\alpha$ is $|\alpha|$-generated.
 \end{enumerate}
\end{lemma}
\begin{proof}
Follows by induction on $\beta < \kappa^+$ using that $\lambda, \lambda_X < \kappa$ and that:
\begin{itemize}
    \item $M_\beta = M_\alpha \oplus \left(\bigoplus_{\nu < \lambda}N_\nu\right)$ if $\beta=\alpha+1$ for $\alpha+1 \notin S$;

    \item $M_\beta \cong \Coker (\iota_\alpha) \oplus L$ if $\beta = \alpha+1$ for some $\alpha \in S$, and

    \item $M_\beta = \bigcup_{\alpha < \kappa}M_\alpha$ if $\beta$ is limit.\qedhere
\end{itemize}
 \end{proof}

 \smallskip

Combining the above with Proposition \ref{p:Strongly free are trivial} and Lemma \ref{lemma: separable implies strongly free} we conclude that $M:=\bigcup_{\alpha<\kappa^+} M_\alpha$ is a $\kappa^+$-generated $(\kappa^+,\Add(X))$-free and $(\kappa^+,\mathrm{Add}(X))$-separable module that is non-trivial, as sought.

\smallskip

A natural inquiry is whether the assumption in Theorem~\ref{MainTheorem} is necessary:
\begin{question}
    Suppose that $\kappa$ is cardinal and that there is a $\kappa^+$-generated non-trivial $(\kappa^+,\Add(X))$-separable module. Must $(\star)_{\kappa,\lambda}$ hold for some $\lambda=\cf(\lambda)<\kappa$? 
\end{question}

Related to this question, we can prove the following.

\begin{proposition}
 Let $\kappa$ be an infinite regular cardinal and $X$, a module with $\lambda_X < \kappa$. Suppose that there is a non-trivial, $\kappa$-generated $(\kappa,\Add_\kappa(X))$-free and strongly $(\kappa,\Add_\kappa(X))$-free module that admits a $\kappa$-filtration $\langle M_\alpha \mid \alpha < \kappa\rangle$ satisfying, for every $\alpha < \kappa$:
 \begin{enumerate}
     \item $M_\alpha \in \Add(X)$,

     \item $M_{\alpha + 1}$ is a direct summand of $M_\beta$ for every $\beta > \alpha+1$, 

     \item $M_\alpha$ is $\leq \lambda_X$-generated if $\alpha< \lambda_X$ and $\leq |\alpha|$-generated if $\alpha \geq \lambda_X$.
 \end{enumerate}
 Then there exists a stationary set $S \subseteq \kappa$ such that, for any regular cardinal $\mu$ with $\lambda_X < \mu < \kappa$, $S$ does not reflect at $\mu$.
\end{proposition}
\begin{proof}
Since $M\notin \Add(X)$,  the set $$S=\{\alpha<\kappa\mid M_\alpha\nleq_{\oplus} M_{\alpha+1}\}$$ %consisting of all ordinals $\alpha < \kappa$ such that $M_\alpha$ is not a direct summand of $M_{\alpha+1}$ 
is, by virtue of Lemma~\ref{WhenMistrivial2}, stationary in $\kappa$. Given any regular cardinal $\mu$ with $\lambda_X < \mu < \kappa$ notice that $S\cap \mu$ is non-stationary. Indeed, $\langle M_\alpha \mid \alpha < \mu\rangle$ serves as  $\mu$-filtration  of $M_\mu$ consisting of modules belonging to $\Add_\mu(X)$, which is (by construction) a $\mu$-generated  trivial $(\mu, \mathrm{Add}(X))$-free module. Thus, by Lemma~\ref{WhenMistrivial2}, $S\cap \mu\in \mathrm{NS}_\mu$. %the set $E \cap \mu$ is not stationary by Proposition \ref{p:WeaklyInAddX}.
\end{proof}
The next is an outright consequence of the previous theorem which, in particular, establishes the compactness result \cite[Theorem IV.3.2]{EklofMekler} in our setting of almost free modules relative to the class $\Add(X)$.

\begin{corollary}
Let $X$ be a module and assume the following:
\begin{quote}
$(\star)$ There exists a regular cardinal $\kappa > \lambda_X$ such that every stationary set $S\s \kappa$ reflects at a regular cardinal $\mu$ satisfying $\lambda_X < \mu < \kappa$.
\end{quote}
Then every $\kappa$-generated $(\kappa,\Add(X))$-free and strongly $(\kappa,\Add(X))$-free module with a filtration satisfying (1), (2) and (3) of the preceding proposition is trivial.

In particular, this is true if $\kappa$ is a weakly compact cardinal.
\end{corollary}

\begin{proof}
    Follows from the preceding proposition. The last assertion is a consequence of the fact that if $E$ is a stationary subset of a weakly compact cardinal, there exists a stationary set $T$ of regular cardinals such that $S$ reflects at $\alpha$ for each $\alpha \in T$, see \cite[Lemma IV.3.1]{EklofMekler}.
\end{proof}

%\begin{question}
 %   Suppose that $\kappa$ is a supercompact cardinal and $X$ is ${<}\kappa$-generated. Is there a $\kappa$-generated non-trivial $(\kappa,\Add(X))$-separable module?
%\end{question}

\section{Perfect decompositions and Enoch's conjecture}\label{sec: applications}

For a module $M$ a \textit{local direct summand} is a submodule $K\leq M$ which can be expressed as $\bigoplus_{i \in I}K_i$ for a family of sumodules $\{K_i\mid i \in I\}$ satisfying that for each finite set $J \subseteq I$, {$\bigoplus_{j \in J}K_j$} is a direct summand of $M$. A natural \emph{compactness-type}\footnote{Compactness is the abstract phenomenon by which the local properties of a mathematical structure determine the global properties of the structure. More specifically, if $\mathfrak{A}$ is a structure of cardinality $\kappa$ and every substructure $\mathfrak{B}$ of size ${<}\kappa$ witness a property $\varphi$ then so does $\mathfrak{A}$ itself.} question is whether every local direct summand of a module $M$ is in fact a direct summand of it. The next notion emerges from this speculation:
\begin{definition}\label{PerfectDecomposition}
A module $X$ is said to have \textit{a perfect decomposition} if for each $M\in \Add(X)$ every local direct summand of $M$ is  a direct summand of it. %$K$ of a module   is a direct summand of $M$.
\end{definition}

The relationship between having a perfect decomposition and the existence of non-trivial $\Add(X)$-separable modules was established in \cite{CortesGuilTorrecillas}. As a consequence of \cite[Corollary 5.13]{CortesGuilTorrecillas}, and assuming the generalized continuum hypothesis, if $X$ has a perfect decomposition, then every $\kappa$-generated and $(\kappa,\Add(X))$-separable module is trivial for every uncountable cardinal $\kappa$ satisfying that $X$ is $<\kappa$-presented (notice that if $X$ has a perfect decomposition, then $X$ satisfies i) of \cite[Corollary 3.13]{CortesGuilTorrecillas} by \cite[Theorem 1.4]{AngeleriSaorin}). In this paper, using Theorem \ref{MainTheorem}, we prove the coverse of this result, that is, we demonstrate that non-trivial almost free modules relative to $\Add(X)$ exist precisely when $X$ does not have a perfect decomposition. Indeed, there  is a natural expectation for this;  on one hand, if $X$ does have a perfect decomposition then $\mathrm{Add}(X)$ ``is compact'' in regards to direct summands; on the other hand, if there is a non-trivial almost free module relative to $\Add(X)$ then ``compactness fails''.

\smallskip

In order to get the result, we assume that there exist a proper class of cardinals $\kappa$ satisfying $(\star)_{\kappa,\lambda}$ for every regular $\lambda < \kappa$. As a consequence of Lemma \ref{l:StarIsConsistent}, this assumption is relatively consistent with ZFC.

\begin{theorem}\label{thm:perfectdecom}
Assume that there exists a proper class of cardinals $\kappa$ satisfying $(\star)_{\kappa,\lambda}$ for each regular $\lambda < \kappa$. The following are equivalent for a module $X$:
\begin{enumerate}
\item $X$ has a perfect decomposition.

\item For every uncountable regular cardinal $\kappa > \lambda_X$ every $\kappa$-generated and $(\kappa,\Add(X))$-free module is trivial.

\item For every uncountable regular cardinal $\kappa > \lambda_X$ every $\kappa$-generated and strongly $(\kappa,\Add(X))$-free module is trivial.

\item For every infinite regular cardinal $\kappa > \lambda_X$ every $\kappa$-generated and $(\kappa,\Add(X))$-separable module is trivial. 

\item For every uncountable regular cardinal $\kappa > \lambda_X$ every $\kappa$-generated, $(\kappa,\Add(X))$-free and $(\kappa,\Add(X))$-separable module is trivial. 
\end{enumerate}
\end{theorem}
\begin{proof}
(1) $\Rightarrow$ (2): Follows from Proposition \ref{prop: filtration and freeness} and \cite[Theorem 1.4]{AngeleriSaorin}.

(1) $\Rightarrow$ (3): Follows from Lemma \ref{l:equivalencebetweenseparability} and \cite[Theorem 1.4]{AngeleriSaorin}.

(3) $\Rightarrow$ (4): Follows from Lemma \ref{lemma: separable implies strongly free}.

(2) $\Rightarrow$ (5) and (4) $\Rightarrow$ (5) are trivial.

(5) $\Rightarrow$ (1):  Suppose that $X$ does not have a perfect decomposition and let a module  $L\in \Add(X)$ admitting a local direct summand  $N:=\bigoplus_{i \in I} \bar{N}_i$ that is not a direct summand of it. Moreover, this is chosen so that $I$ has the minimum cardinality witnessing this. Set $|I|=\lambda$. It is not hard to show that $\lambda$ is regular. Writing $I=\bigcup_{\alpha < \lambda}I_\alpha$ for a family of sets $\{I_\alpha \mid \alpha < \lambda\}$ with cardinality smaller than $\lambda$ and setting $N_\beta:=\sum_{i \in I_\alpha} \bar{N}_\alpha$ we get a continuous chain of direct summands of $L$ with union $N$. In particular, $\langle N,L, \mathcal{N}\rangle$ is a $\lambda$-template, where $\mathcal N=\langle N_\alpha \mid \alpha < \lambda\rangle$.

\smallskip

Now let $\kappa$ be a cardinal greater than $\max\{\lambda_X,\lambda\}$ and satifying $(\star)_{\kappa,\lambda}$. We can apply Theorem~\ref{MainTheorem} to get a $\kappa^+$-generated, $(\kappa^+,\Add(X))$-free and $(\kappa^+,\Add(X))$-separable module that is not trivial. Then (5) above is false. %This concludes the proof.%\ale{Maybe $\kappa^+$-separable?}
\end{proof}
%An interesting aspect of the above theorem is that it connects two instances of compactness provided the set-theoretic universe encompasses many non-compact objects (i.e., non-reflecting stationary sets).\smallskip

Changing gears next we establish the consistency of ZFC with Enoch's conjecture for classes of the form $\Add(X)$ (see p.\pageref{sec: prelimminaries}).
%Employing the set-theoretic assumption used in Theorem \ref{thm:perfectdecom} entails Enoch's conjecture  
This result was first proved by J. \v Saroch in \cite[Theorem 2.2]{Saroch} under the existence of a proper class of cardinals $\kappa$ satisfying that every stationary subset $E$ of $\kappa$ admits a non-reflecting stationary subset. In our case we will employ the set-theoretic assumption used in Theorem \ref{thm:perfectdecom}.  

\smallskip

We will use the following result, which is based on \cite[Corollary 5.3]{BSP} which states that a module $M$ belongs to $\mathcal{X}$ provided it has a  $\mathcal X$-precover with locally split kernel and an $\mathcal X$-cover (here $\mathcal X$  is a class of modules closed under direct summands). Recall that a morphism $m:K \rightarrow L$ is called locally split if for every $k \in K$ there exists $h:L \rightarrow K$ with $hm(k)=k$.

\begin{lemma}\label{ExistenceofAddXcovers}
    Let $\kappa$ be an infinite regular cardinal, $X$ a module and $M$, a $\kappa$-generated and $(\kappa, \mathrm{Add}(X))$-separable module with $\kappa>\lambda_X$. Then the following assertions are equivalent:
    \begin{enumerate}
        \item $M\in \mathrm{Add}(X)$;
        \item $M$ has an $\Add(X)$-cover.
    \end{enumerate}
   Thus, if $\Add(X)$ is covering then every $\kappa$-generated and $(\kappa,\Add(X))$-separable module is trivial for all regular cardinals $\kappa>\lambda_X$.
\end{lemma}
\begin{proof}
  Let $\langle M_\alpha\mid \alpha<\kappa\rangle$ be a $\kappa$-filtration of $M$ given by Lemma \ref{l:equivalencebetweenseparability} and $$\textstyle \varphi\colon \bigoplus_{\alpha<\kappa} M_\alpha\rightarrow M$$ be the homomorphism induced by the inclusion maps $i_\alpha\colon M_\alpha\rightarrow M$.  Since $\varphi$ is the morphism associated to a totally ordered direct limit, it is locally split by \cite[Lemma 2.1]{GomezGuil}. Since $X$ is ${<}\kappa$-generated and $M$ is $(\kappa,\Add(X))$-separable, $\varphi$ is an $\{X\}$-precover, hence an $\Add(X)$-precover. By \cite[Corollary 5.3]{BPS}, $M\in \mathrm{Add}(X).$
\end{proof}

%Our next theorem will rely on the following set-theoretical hypothesis:
%\begin{quote}\label{h:Stationary}
%$(\mathcal{S})^+$: For each infinite cardinal $\kappa$ and a stationary set $S\s \kappa^+$ there is a stationary set $T\s S$  which does not reflect and $T\in I[\kappa^+]$.% a non-reflecting stationary set. % $S'$ which is stationary and non-reflecting in $\kappa^+$.
%\end{quote}
%Once again
 %   $(\mathcal{S})^+$ is consistent with ZFC as it is implied, e.g. by $``0^\sharp$ does not exist''.

\begin{theorem}
Assume that there exists a proper class of cardinals $\kappa$ satisfying $(\star)_{\kappa,\lambda}$ for each regular  $\lambda < \kappa$. The following are equivalent for a module $X$:
\begin{enumerate}
\item $X$ has a perfect decomposition;
\item $\Add(X)$ is closed under direct limits;
\item $\mathrm{Add}(X)$ is a covering class;
\end{enumerate}
    Therefore, under our set-theoretic assumption Enoch's conjecture holds.
\end{theorem}
\begin{proof}
    (1) $\Rightarrow$ (2). By \cite[Theorem 2.1]{AngeleriSaorin} and Lemma \ref{lemma: prelimminaries well-ordered limits}.

    (2) $\Rightarrow$ (3). The class $\Add(X)$ is always precovering. If it is closed under direct limits, then it must be covering (see e.g., \cite[Theorem 2.2.12]{Xu}).

    (3) $\Rightarrow$ (1) Follows from the preceding lemma and Theorem \ref{thm:perfectdecom}.
\end{proof}

Our forthcoming theorems will not bear on the instrumental Theorem~\ref{MainTheorem} yet will still
draw further connections between almost free modules and Enoch's conjecture. Part of the subsequent discussions bear on the following concept:

\begin{definition}
    Let $\kappa$ be an infinite regular cardinal and $\mathcal X$ a class of modules. We say that $\mathcal X$ is closed under $\kappa$-free modules if every $(\kappa,\mathcal X)$-free module belongs to $\mathcal X$.
\end{definition}

A natural direction to asses the independence of Enoch's conjecture from ZFC is to find (perhaps, subject to extra set-theoretic hypothesis) a module $X$ such that the following hold: (1) $\Add(X)$ is not closed under direct limits; (2) every $(\kappa,\Add(X))$-separable module is trivial for every $\kappa >\lambda_X$. In view of Lemma \ref{ExistenceofAddXcovers}, $\Add(X)$ would be a candidate for a class not satisfying Enoch's conjecture. The following result, established in ZFC, says that for $\kappa=\aleph_1$ this is not possible:

 \begin{theorem}\label{thm: Enochs for X}
    Suppose that $\mathcal{X}$ is a class of modules closed under $\aleph_1$-free modules, direct sums and direct summands. Then, if $\mathcal{X}$ is precovering it is closed under direct limits. In particular, Enoch's conjecture holds for those classes.
\end{theorem}
\begin{proof}
    By Lemma~\ref{lemma: prelimminaries well-ordered limits} it suffices to show that $\mathcal{X}$ is closed under direct limits of totally ordered systems of modules in $\mathcal{X}$. Since every totally ordered set admits a cofinal subset that is linearly ordered we may and do assume that our system of modules $\mathcal{F}$ is indexed by an ordinal $\lambda$; thus $\mathcal{F}=(F_{\alpha},\iota_{\alpha,\beta})_{\alpha\leq \beta<\lambda}$ with $F_\alpha \in \mathcal X$ for every $\alpha < \kappa$. By passing to a cofinal subset of $\lambda$  we may assume  
   % Both  $\mathcal{D}$ and  $\langle F_{\alpha_\nu},\iota_{\alpha_\nu,\alpha_{\eta}}\mid \nu\leq \eta<\cf(\mu)\rangle$ -- for $\langle \alpha_\nu\mid \nu<\cf(\mu)\rangle$  a cofinal sequence in $\mu$ -- yield the same direct limit. Thus, 
   that $\lambda$ is a regular cardinal.  Set $F:=\varinjlim \mathcal{F}$. 

   \smallskip

   Using the argument of \cite[Lemma 2.3]{Cortes17}, we can find an infinite cardinal $\kappa$ greater than $\max\{\lambda, |F|, |R|\}$ satisfying that $\kappa^\mu=\kappa$ for each cardinal $\mu < \kappa$ and $\kappa^\lambda=2^\kappa$. Following \cite[Section~2]{Cortes17}, one constructs the $(\aleph_1,\mathrm{Sum}\{F_\alpha\mid \alpha<\lambda\})$-free module $L$ associated to  $\mathcal{F}$ and $\kappa$. Since $\mathcal{X}$ is closed both under direct sums and $\aleph_1$-free modules it follows that $L\in \mathcal{X}.$ Also, by \cite[Lemma~2.1]{Cortes17} this module comes together with a short exact sequence 
  
\begin{displaymath}
\begin{tikzcd}
0 \arrow{r} & D \arrow{r}{\s}  & L \arrow{r}  & F^{(2^\kappa)} \arrow{r} & 0.
\end{tikzcd}
\end{displaymath}
where $|D|=\kappa$ by the election of $\kappa$. Since $\mathcal{X}$ is a precovering class there is a  short exact sequence
\begin{displaymath}
\begin{tikzcd}
0 \arrow{r} & M \arrow{r}{m}  & A \arrow{r}{f}  & F \arrow{r} & 0
\end{tikzcd}
\end{displaymath}
with $A\in\mathcal{X}$ and  $f$ an $X$-precover of $F$. 
%Next we show that $F$ is a direct summand of $A$:
\begin{claim}
    $f$ splits.
\end{claim}
\begin{proof}[Proof of claim]
    This argument is due to \v{S}aroch (see \cite[Lemma~3.2]{SarochMittagLeffler}) and we include it for the reader's convenience only. Apply $\mathrm{Hom}_{R}(-, m)$ to the previously displayed short exact sequence to obtain the following diagram with exact arrows
    \begin{displaymath}
\begin{tikzcd}
\mathrm{Hom}_{R}(D, M) \arrow{r}{\delta} \arrow{d}{\mathrm{Hom}_{R}(D, m)} & \mathrm{Ext}^1_R(F,M)^{2^\kappa} \arrow{r} \arrow{d}{\mathrm{Ext}^1_R(F,m)^{2^\kappa}} & \mathrm{Ext}^1_R(L,M)   \arrow{d}{\mathrm{Ext}^1_R(L,m)} \\
\mathrm{Hom}_{R}(D, A) \arrow{r} & \mathrm{Ext}^1_R(F,A)^{2^\kappa} \arrow{r}{\overline k} & \mathrm{Ext}^1_R(L,A). 
\end{tikzcd}
\end{displaymath}
Since $L\in \mathcal{X}$ and $f$ is an $\mathcal{X}$-precover, $\mathrm{Ext}^1_R(L,m)$ is monic, hence $\mathrm{Ker}(\mathrm{Ext}^1_R(F,m))^{2^\kappa}\s \mathrm{Im}(\delta)$. Nonetheless, $2^\kappa\geq |\mathrm{Hom}_R(D,M)|\geq |\mathrm{Im}(\delta)|$ and if $\mathrm{Ext}^1_{R}(F,m)$ is not monic, then $\Ker(\mathrm{Ext}^1_R(F,m))^{2^\kappa})\geq 2^{2^\kappa}$, so that $\Ext^1_R(F,m)$ has to be monic. This latter is equivalent to $\mathrm{Hom}_R(F,f)$ to be onto which implies, in particular, that $f$ splits.
\end{proof}

Thus $F\leq_{\oplus} A\in \mathcal{X}$ and, since $\mathcal X$ is closed under direct summands $F\in \mathcal{X}$.
\end{proof}

We are going to apply this result to some classes of relatively Mittag-Leffler modules. For a class $\mathcal{Q}$ of right $R$-modules, a left $R$-module $M$ is called \emph{$\mathcal{Q}$-Mittag-Leffler} if for each family of modules $\{ Q_i\mid i\in I\}$ belonging to $\mathcal{Q}$ the natural homomorphism
$$\psi\colon M\otimes_R \prod_{i\in I} Q_i\rightarrow \prod_{i\in I}M\otimes_R Q_i$$
is monic. A  left $R$-module $M$ is called \emph{Mittag-Leffler} whenever $\mathcal{Q}=\mathrm{Mod}$-$R$.

Mittag-Leffler modules were introduced by Raynaud
and Gruson in \cite{Rayunaud}. The key property of the class of $\mathcal Q$-Mittag-Leffler modules is that it is closed under $\aleph_1$-free modules \cite[Theorem 2.6]{HerberaTrlifaj}. Actually, the class of all modules with $\mathcal Q$-Mittag-Leffler dimension less than or equal to $n$ is closed under $\aleph_1$-free modules \cite[Corollary 3.4]{CortesProducts}. Recall that a module $M$ has \textit{$\mathcal Q$-Mittag-Leffler dimension less than or equal to $n\geq 1$} if $M$ has a projective resolution whose $(n-1)st$-syzygy is $\mathcal Q$-Mittag-Leffler (and we say that $M$ has $\mathcal Q$-Mittag-Leffler dimension $0$ if it is $\mathcal Q$-Mittag-Leffler). We say that the left global $\mathcal Q$-Mittag-Leffler dimension of $R$ is less than or equal to $n$ if every module has $\mathcal Q$-Mittag-Leffler dimension less than or equal to $n$. 

\smallskip

As an immediate consequence of Theorem~\ref{thm: Enochs for X}  we get the following extension of \cite[Theorem 2.6]{TrlifajYassine} to the non-flat setting:

\begin{corollary}
    Let $\mathcal Q$ be a class of right $R$-modules, $n$ a natural number. If $\mathcal M$ is the class of all left $R$-modules with $\mathcal Q$-Mittag-Leffler dimension less than or equal to $n$ then, $\mathcal M$ is precovering if and only if the left global $\mathcal Q$-Mittag-Leffler of $R$ is less than or equal to $n$.

    In particular, Enoch's conjecture holds for $\mathcal M$.
\end{corollary}

\begin{proof}
    Suppose that $\mathcal M$ is precovering. The class $\mathcal M$ is closed under direct sums, direct summands and $\aleph_1$-free modules by \cite[Corollary 3.4]{CortesProducts}. By the previous theorem, $\mathcal M$ is closed under direct limits. But this implies that every module belongs to $\mathcal M$, since every module is a direct limit of finitely presented modules and finitely presented modules are $\mathcal Q$-Mittag-Leffler.
\end{proof}

Given $n$ a natural number, we say that the ring $R$ is \textit{left weak $n$-coherent} \cite[p. 4566]{CortesProducts} if the product of every family of flat right $R$-modules has flat dimension less than or equal to $n$ ($R$ is left weak $0$-coherent if and only if it is left cohererent). We get:

\begin{corollary}
    Let $\mathcal F$ be the class of all flat right $R$-modules, $n$ a natural number and $\mathcal M_n$ the class of all left  $R$-modules with $\mathcal F$-Mittag-Leffler dimension less than or equal to $n$. Then:
    \begin{enumerate}
        \item If $n=0$, $\mathcal M_0$ is precovering if and only if $R$ is left noetherian.
        \item If $n=1$, $\mathcal M_1$ is precovering if and only if $R$ is left coherent.
        \item If $n>1$, $\mathcal M_n$ is precovering if and only if $R$ is left weak $n$-coherent.
    \end{enumerate}
\end{corollary}

\begin{proof}
    (2) and (3) follows from \cite[Theorem 4.2]{CortesProducts}. For $n=0$, the previous corollary says that every left $R$-module is $\mathcal F$-Mittag-Leffler. This is equivalent to $R$ being left noetherian by \cite[Example 5.6]{AngeleriHerbera}.
\end{proof}

\bibliographystyle{alpha} 
\bibliography{citations}
\end{document}